\pdfoutput=1

\documentclass{amsart}
\usepackage{fullpage }   

\usepackage[centertags]{amsmath}
\usepackage{amsfonts}
\usepackage{amssymb}
\usepackage{amsthm}
\usepackage{ subfigure}
\usepackage{graphicx}
\usepackage{color}
\usepackage{enumerate}

\newtheorem{theorem}{Theorem}[section]
\newtheorem{proposition}[theorem]{Proposition}
\newtheorem{lemma}[theorem]{Lemma}

\newtheorem{corollary}[theorem]{Corollary}

\theoremstyle{definition}
\newtheorem{definition}[theorem]{Definition}

\newtheorem{example}[theorem]{Example}

\theoremstyle{remark}
\newtheorem{remark}[theorem]{Remark}

\numberwithin{equation}{section}

\begin{document}


\newcommand{\BR}{Bollob\'{a}s-Riordan }
\newcommand{\MBR}{topochromatic }

\newcommand{\V}{V}
\newcommand{\E}{{E}}

\newcommand{\BG}{{ G}}

\newcommand{\bv}{{ v}}
\newcommand{\BV}{{ V}}
\newcommand{\BE}{{ E}}
\newcommand{\BC}{\mathrm{BC}}

\newcommand{\be}{{ e}}
\newcommand{\bb}{{ \boldsymbol b}}
\newcommand{\bu}{{  u}}

\newcommand{\bal}{{\boldsymbol\alpha}}
\newcommand{\bbe}{{\boldsymbol\beta}}
\newcommand{\bga}{{\boldsymbol\gamma}}

\newcommand{\al}{ \alpha}
\newcommand{\bet}{ \beta}
\newcommand{\ga}{ \gamma}
\newcommand{\la}{ \lambda}

\newcommand{\hZ}{\tilde{Z}}

\newcommand{\g}{\xi}
\newcommand{\h}{\zeta}

\newcommand{\vg}{\boldsymbol\xi}
\newcommand{\vh}{\boldsymbol\zeta}

\newcommand{\mm}{\bar{m}}
\newcommand{\cG}{G_c}
\newcommand{\cmG}{G_{\bar{c}}}
\newcommand{\fG}{\mathfrak{G}}
\newcommand{\fS}{S}
\newcommand{\calG}{\mathcal{G}}
\newcommand{\calS}{\mathcal{S}}

\newcommand{\vs}{\vec{s}\,{}^{\prime} }
\newcommand{\vG}{\vec{s}\,{}^{\Gamma} }


\title[ Twisted duality for embedded graphs ]{Twisted duality for embedded graphs }



\author[J.~Ellis-Monaghan]{Joanna A. Ellis-Monaghan}
\address{Department of Mathematics, Saint Michael's College, 1 Winooski Park, Colchester, VT 05439, USA.  }
\email{jellis-monaghan@smcvt.edu}
\thanks{The work of the first author was supported by the National Science Foundation (NSF) under grant number DMS-1001408, by the Vermont Space Grant Consortium through the National Aeronautics and Space Administration (NASA), and by the Vermont Genetics Network through Grant Number P20 RR16462 from the INBRE Program of the National Center for Research Resources (NCRR), a component of the National Institutes of Health (NIH).  This paper's contents are solely the responsibility of the authors and do not necessarily represent the official views of the NSF, NASA, NCRR, or NIH}

\author[I.~Moffatt]{ Iain Moffatt}
\address{Department of Mathematics and Statistics,  University of South Alabama, Mobile, AL 36688, USA}
\email{imoffatt@jaguar1.usouthal.edu}


\date{\today}

\begin{abstract}  
We consider two operations on an edge of an embedded graph (or equivalently a ribbon graph): giving a half-twist to the edge and taking the partial dual with respect to the edge.   These two operations give rise to an action of ${\fS_3}^{|E(G)|}$, the \emph{ribbon group},  on $G$. 
The action of the ribbon group on embedded graphs extends the concepts of duality, partial duality and Petrie duality.
We show that this ribbon group action gives a complete characterization of duality in that if $G$ is any cellularly embedded graph with medial graph $G_m$, then the orbit of $G$ under the group action is precisely the set of all graphs with medial graphs isomorphic (as abstract graphs) to $G_m$. We provide characterizations of special sets of twisted duals, such as the partial duals, of embedded graphs in terms of medial graphs and we show how different kinds of graph isomorphism  give rise to these various notions of duality. The ribbon group action then leads to a deeper understanding of the properties of, and relationships among, various graph polynomials  via the generalized transition polynomial which interacts naturally with the ribbon group action.  

\end{abstract}

\maketitle

\section{Introduction}

We define the \emph{ribbon group action} on the set of embedded graphs by taking duals and twisting individual edges. The ribbon group action completes the classical concept of duality for embedded graphs, and we exploit this group action to  unravel and explore connections between embedded graphs and  their medial graphs, and to determine implications of these new relations.  The ribbon group action allows us to: classify the twisted duals of an embedded graph; characterize all graphs with isomorphic medial graphs under various kinds of  isomorphism; extend the concept of a Tait graph to all $4$-regular graphs;  and to determine new properties of polynomials of embedded graphs.

\medskip

There are many investigations into planar and abstract graphs that have not yet been fully extended to the context of graphs embedded in surfaces. Consider for example the classical relations among the medial graphs and the duals of plane graphs. Suppose that $G$ is a plane graph, $G^*$ is its dual, and $G_m$ is its medial graph. The medial graph of $G^*$ is exactly the medial graph of $G$, \emph{i.e.} $(G^*)_m=G_m$, where ``$=$'' denotes equality as plane graphs. In fact, the connection between geometric duals and medial graphs is a little stronger than this. The two graphs 
 $G$ and $G^*$ are the only plane graphs that have $G_m$ as their plane medial graphs, that is,
\begin{equation}\label{mot}  \{G,G^*\} = \{H\;|\; H_m=G_m\}. \end{equation} 
This equality of sets also holds when $G$ is cellularly embedded in any surface, not just the plane.

How does the left-hand set in (\ref{mot}) change if, for example, $G_m$ and $H_m$ are equivalent, not as plane embeddings, but as abstract graphs?  How then are $G$ and $H$ related? The ribbon group action provides a way to answer this question precisely.  We will see that   $G_m$ and $H_m$ are equivalent as abstract graphs if and only if $G$ and $H$ are in the same orbit under the ribbon group action (see Theorem~\ref{c.orbg}). 
In fact, the ribbon group action answers a host of related isomorphism questions. For example,  we will determine the set of embedded graphs whose medial graphs are isomorphic as abstract graphs to a given $4$-regular graph $F$, in particular, the set $\{ G \;|\; G_m \cong F  \}$ is exactly the set of cycle family graphs of an arbitrary embedding of $F$ (see Theorem~\ref{cycle to medial}).

\medskip

The ribbon group action, introduced in Section~\ref{s.td}, is a far-reaching extension of the notions of  geometric duality and  Petrie duality.  
Chmutov, in \cite{Ch1}, extended the idea of a geometric dual of an embedded graph by describing how a dual can be formed   with respect to individual edges of the embedded graph. The resulting duality, called partial duality, has found a number of applications in mathematics and physics (see, for example, \cite{Ch1,KRVb, Mo5} and the references therein). 
 Here we take the  idea of a partial dual further by allowing two operations on the edges of an embedded graph $G$: taking the dual with respect to an edge, and also giving a half-twist to an edge. These two operations give rise to a group action of ${S_3}^{|E(G)|}$ on $G$, which we call the \emph{ribbon group action}. We say that two embedded graphs are {\em twisted duals} if and only if they lie in the same orbit under the group action. Chmutov's partial duality and Petrie duality arise as the action of  important, natural subgroups of the ribbon group.

Many of the applications of twisted duality presented in this paper arise from the connection between an embedded graph and its embedded medial graph. Again we are motivated by classical results in graph theory. If $F$ is a $4$-regular plane graph, then it is necessary checkerboard colourable, and it is well-known how to construct the set $\{ G\;|\; G_m=F  \}$ of embedded graphs which have $F$  as their medial graph. The graphs in this set are called the Tait graphs of $F$. When $F$ is a non-plane graph, then the construction of the Tait graphs can fail since not all embedding of a $4$-regular graph may be checkerboard colourable. However, in Section~\ref{s.medialandribbongroup} we introduce {\em cycle family graphs}. Cycle family graphs extend of  the idea of a Tait graph to arbitrary $4$-regular embedded graphs.  It is well known that the Tait graphs of a $4$-regular checkerboard colourable graph are geometric duals. We show that this classical connection between graphs and medial graphs lifts to cycle family graphs and twisted duals. In Subsection~\ref{ss.mgpd}, we show that the set of twisted duals of an embedded graph $G$ is exactly the set of cycle family graphs of its medial graph $G_m$.
This result provides a powerful connection between embedded graphs and their medial graph and is one of the main tools in our paper.

Twisted duality leads to the discovery that the concepts of duality and of graph equality are equivalent. Revisiting the identity (\ref{mot}) above, we observe that definition of equality of medial graphs determines, and is determined by, the definition of duality: if we begin with characterizing the set  $\{H\;|\; H_m=G_m\}$, we are led to the classical concept of geometric duality, and vice versa.  This suggests the following point of view:  if $\sim$ denotes some notion of equivalence of embedded graphs, then $\{ H\; |\;  H_m\sim G_m\}$ is exactly the  set of graphs dual to $G$ for some corresponding notion of duality. Thus, we can regard notions of equality of  embedded graphs as being equivalent to notions of duality. This point of view provides a hierarchy and a framework for understanding generalizations of duality, and for understanding isomorphism problems for $4$-regular graphs. We show how geometric duality, partial duality and twisted duality fit into this framework and provide corresponding notions of embedded graph equivalence. In particular, in Subsection~\ref{ss.orbits}, we use twisted duality to solve isomorphism problems for medial graphs by proving that if $G$ and $H$ are embedded graphs then:
\begin{itemize}
\item $G_m$ and $H_m$ are equivalent as embedded graphs if and only if $G$ and $H$ are geometric duals;
\item  $G_m$ and $H_m$ are equivalent as locally embedded maps if and only if $G$ and $H$ are partial duals;
\item $G_m$ and $H_m$ are equivalent as abstract graphs if and only if $G$ and $H$ are twisted duals.
\end{itemize}

\medskip

Twisted duality also proves useful for studying graph polynomials.  Graph polynomials are a natural forum for studying duality properties.  However, most graph polynomials apply either to plane graphs ({\em e.g.} the Penrose polynomial of \cite{Pen71}) or to abstract graphs ({\em e.g.} the classical Tutte polynomial of \cite{Tut47,Tut54,Tut67}).  Notable exceptions are the extensions of the Tutte polynomial to embedded graphs by Las Vergnas in \cite{Las78}, and by Bollob\'as and Riordan in \cite{BR1, BR2}. In Section~\ref{sec:transpoly}, we adapt the generalized transition polynomial of \cite{E-MS02} to embedded graphs.    We then showcase the applicability of twisted duality and the ribbon group action by showing that the ribbon group action corresponds to an action of the symmetric group on the vertex state weights of the transition polynomial.  This gives a twisted duality relation in terms of the weight systems for this polynomial. Because the transition polynomial  assimilates or is related to many other graph polynomials, this then leads to twisted duality relations for other polynomials. 

In particular, an extension of the Penrose polynomial to ribbon graphs is given in \cite{E-MMc}, and, via the transition polynomial, the theory of twisted duality leads to new properties of the Penrose polynomial, including a restatement of the Four Colour Theorem.  The transition polynomial is also related to the topological Tutte polynomial of  Bollob\'as and Riordan \cite{BR1, BR2}, so the identities for the transition polynomial derived in Section~\ref{sec:transpoly} using twisted duality now transfer to this polynomial and this leads to a number of new results about the polynomial (see \cite{E-MMd}).

\section{Embedded graphs} \label{Sec Embedded graphs}

We use the term ``embedded graph'' loosely to mean any of four equivalent representations of graphs in surfaces.  We may think of an embedded graph as any of:
\begin{enumerate}\renewcommand{\labelenumi}{(\alph{enumi})}
\item a cellularly embedded graph, that is, a graph embedded in a surface such that every face is a 2-cell;
\item a signed rotation system;
\item a ribbon graph, also known as a band decomposition;
\item an arrow presentation.
\end{enumerate}    

In subsequent sections we will use these equivalent representations interchangeably, using whichever best facilitates the discussion at hand.  We assume the reader is familiar with cellular embeddings of graphs and signed rotation systems, recalling that a signed rotation system consists of an abstract graph $G$, a cyclic ordering of the half-edges at each vertex, and a $+$ or $-$ sign at each edge. Signed rotation systems are considered up to equivalence under vertex flips, that is, reversing the cyclic order at a vertex and simultaneously reversing the signs on the incident edges (twice in the case of a loop).  Under this equivalence, signed rotation systems correspond to cellularly embedded graphs.  An abstract graph with a cyclic ordering of the half-edges at the vertices, but without an assignment of signs to the edges, is also called a \emph{combinatorial map}, and we will use this term for emphasis when we are disregarding the signs.  See Mohar and Thomassen \cite{MT01} for the details of these constructions.  

We briefly review ribbon graphs, arrow presentations, and their equivalence (which is homeomorphism of the surface). We use standard notation: $V(G)$, $E(G)$, and $F(G)$ are the vertices, edges, and faces, respectively, of an embedded graph $G$, while $v(G)$, $e(G)$, and $f(G)$, respectively, are the numbers of such.
We will say that a loop ({\em i.e.}  an edge that is incident to exactly one vertex) in an embedded graph is {\em non-twisted} if a neighbourhood of it is orientable, and we say that the loop is {\em twisted} otherwise.

\begin{definition}
Following \cite{BR2}, a {\em ribbon graph} $G =\left(  \V(G),\E(G)  \right)$ is a (possibly non-orientable) surface with boundary, represented as the union of two  sets of topological discs: a set $\V (G)$ of {\em vertices}, and a set $\E (G)$  of {\em edges} such that 
\begin{enumerate} 
\item the vertices and edges intersect in disjoint line segments;
\item each such line segment lies on the boundary of precisely one
vertex and precisely one edge;
\item every edge contains exactly two such line segments.
\end{enumerate}
\end{definition}

Two ribbon graphs are said to be {\em equivalent} or {\em isomorphic} if there is a homeomorphism between them that preserves the vertex-edge structure.
 (In particular, up to this equivalence, the embedding of a ribbon graph in three-space is irrelevant.) 
 The {\em genus} of a ribbon graph  is its genus as a punctured surface. (The genus of a punctured surface is the genus of the closed surface obtained  by capping-off all of the punctures.)  
 We will say that a ribbon graph is {\em plane} if it is the neighbourhood of a plane graph, or equivalently, if the ribbon graph is a genus zero surface. A ribbon graph is {\em orientable} if it is orientable as a surface.
  Ribbon graphs are  \emph{band decompositions} (see Gross and Tucker  \cite{GT87}), with the 2-band interiors removed (the 2-bands correspond to faces). Note that in this context, $f(G)$ is just the number of boundary components of the surface with boundary comprising the ribbon graph $G$.

\begin{definition}[Chmutov \cite{Ch1}]
An {\em arrow presentation} consists of  a set of circles,  each with a collection of disjoint,  labelled arrows, called {\em marking arrows}, indicated along their perimeters. Each label appears on precisely two arrows.
\end{definition}

  Two arrow presentations are {\em equivalent}  if one can be obtained from the other by reversing the direction of all of the marking arrows which belong to some subset of labels, or by relabelling the pairs of arrows.

The equivalence of these representations of embedded graphs is given rigorously in, for example \cite{GT87}; however it is intuitively clear.  If $G$ is a cellularly embedded graph, a ribbon graph representation results from taking a small neighbourhood  of the embedded graph $G$. This can be thought of as \emph{cutting out} the ribbon graph from the surface.  Neighbourhoods  of vertices of the graph $G$ form the vertices of a ribbon graph, and neighbourhoods of the edges of the embedded graph form the edges of the ribbon graph. 

On the other hand, if $G$ is a ribbon graph, we simply sew discs into each boundary component of the ribbon graph to get the desired surface. See Figure~\ref{ribbon 2-cell}, which shows a graph embedded in the projective plane.

\begin{figure}
\[\includegraphics[width=40mm]{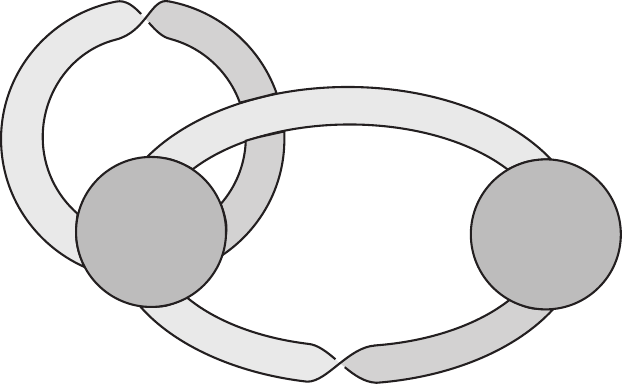} \quad
 \raisebox{10mm}{\includegraphics[width=18mm]{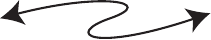}}\quad
\raisebox{0mm}{ \includegraphics[width=40mm]{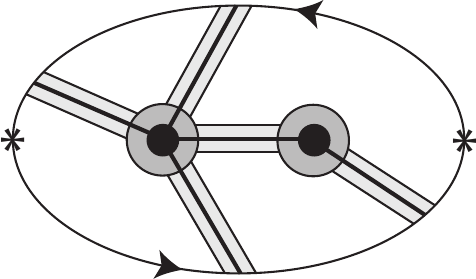}}\quad
 \raisebox{10mm}{\includegraphics[width=18mm]{doublearrow}}\quad
\raisebox{0mm}{ \includegraphics[width=40mm]{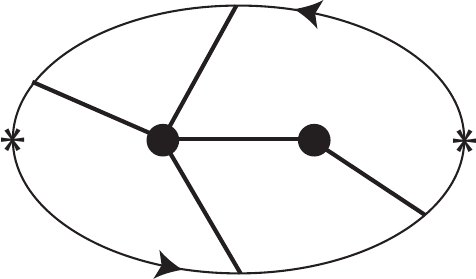}}.\]
\caption{Equivalence of ribbon graphs and 2-cell embeddings.}
\label{ribbon 2-cell}
\end{figure}

A ribbon graph can be obtained from an arrow presentation by viewing each circle as the boundary of a disc that becomes a vertex of the ribbon graph.  Edges are then added to the vertex discs by taking a disc for each label of the marking arrows.  Orient the edge discs arbitrarily and choose two non-intersecting arcs on the boundary of each of the edge discs. Orient these arcs according to the orientation of the edge disc.
Finally, identify these two arcs with two marking arrows, both with the same label, aligning the direction of each arc consistently with the orientation of the marking arrow. See Figure~\ref{arrows}.

\begin{figure}
\begin{tabular}{p{4cm}cp{4cm}cp{4cm}}
\includegraphics[height=15mm]{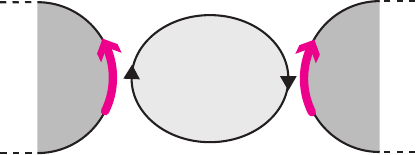}  &
\raisebox{6mm}{\hspace{3mm}\includegraphics[width=11mm]{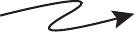}}&
\includegraphics[height=15mm]{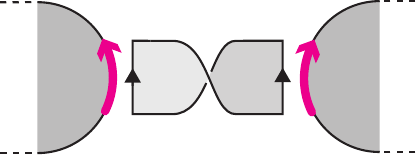} &
\raisebox{6mm}{\hspace{3mm}\includegraphics[width=11mm]{arrow}}&
 \includegraphics[height=15mm]{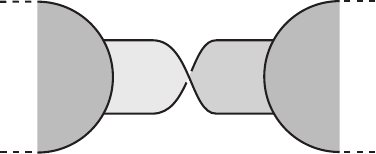}  \\
Two vertices and an edge disc && aligning the arrows &  & identifying arcs 
\end{tabular}
\caption{Constructing a ribbon graph from an arrow presentation. }
\label{arrows}
\end{figure}

Conversely, every ribbon graph  gives rise to an arrow presentation.  To describe a ribbon graph $G$ as an arrow presentation, start by arbitrarily orienting and labelling  each edge disc in $\E(G)$. This induces an orientation on the boundary of each edge in $\E(G)$.  Now, on the arc where an edge disc intersects a vertex disc, place a marked arrow on the vertex disc, labelling the arrow with the label of the edge it meets and directing it consistently with the orientation of the edge disc boundary. The boundaries of the vertex set marked with these labelled arrows give the arrow marked circles of an arrow presentation.  See Figure~\ref{ribbon arrow}.

Finally, the relation between signed rotations systems and cellularly embedded graphs is standard and well known.

\begin{figure}
\[ \raisebox{4mm}{\includegraphics[width=35mm]{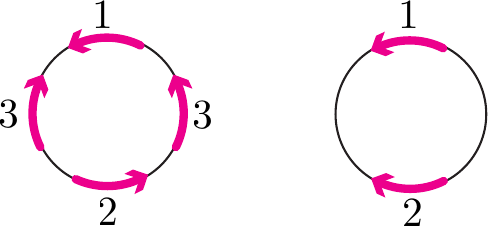}}\quad
\raisebox{10mm}{\includegraphics[width=18mm]{doublearrow}} \quad
 \includegraphics[width=40mm]{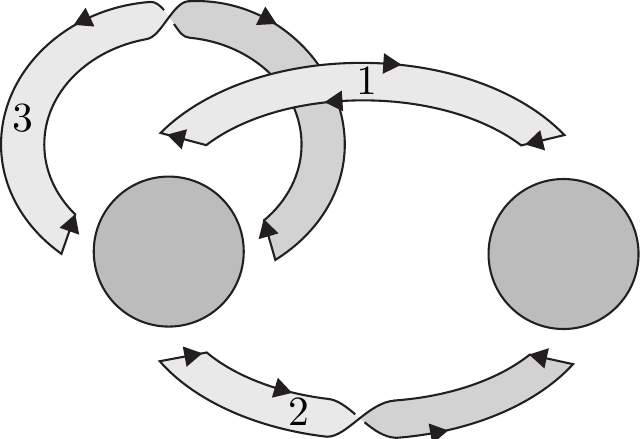} \quad
 \raisebox{10mm}{\includegraphics[width=18mm]{doublearrow}}\quad
\raisebox{1mm}{ \includegraphics[width=40mm]{arrpresexamp3}}.\]
\caption{Equivalence of arrow presentations and ribbon graphs.}
\label{ribbon arrow}
\end{figure}

\bigskip

These ways of representing embedded graphs are equivalent, and furthermore, two graphs are equivalent as ribbon graphs if and only if they are equivalent as cellularly embedded graphs if and only if they are equivalent as arrow presentation if and only if they are equivalent as signed rotation systems.  When two embedded graphs are equivalent in this sense, we will often say that they are {\em equivalent as embedded graphs} or are {\em isomorphic}, to distinguish this kind equivalence from other graph equivalences that arise in this paper  (we will be interested in how different notions of equivalence generate extensions of the concept of the dual of an embedded graph). 

Two other natural notions of equivalence are important here. There is a natural forgetful functor from the category of embedded graphs to the category of graphs which forgets all of the information concerning the embedding. Thus every embedded graph gives rise to an (abstract) graph by forgetting the embedding and retaining only the vertices and their adjacency information. A cellularly embedded graph $G$ consists of a graph $G'$ and a cellular embedding $\psi$ of $G'$ into a surface. We say that (the abstract graph) $G'$ is the  {\em underlying graph} of (the embedded graph) $G$. We will say that two embedded graphs are {\em equivalent as abstract graphs} if their underlying graphs are isomorphic (or equivalent) as abstract graphs.  We will usually denote isomorphism of abstract graphs by $G \cong F$, reserving $G = F$ for when $G$ and $F$ are isomorphic as embedded graphs.

We will also be interested in a notion of equivalence that is stronger than equivalence as abstract graphs, but weaker than equivalence as embedded graphs.
If we are given an embedded graph $G$, then we can arbitrarily assign orientations to  a neighbourhood of each vertex of $G$. By reading off the incident half-edges at each vertex of $G$ according to these orientations, we obtain a cyclic order  of the half-edges incident to each vertex. By equipping the underlying abstract graph of $G$ with this cyclic order we obtain a combinatorial map $M$ (that is, an unsigned rotation system). We call $M$ an {\em  underlying combinatorial map} of $G$, and we call $G$ a \emph{local embedding of $M$}. Of course, in general an embedded graph will have more than one underlying combinatorial map, and a combinatorial map will have more than one local embedding.  It is easily seen, however, that all of the  underlying combinatorial maps  of a given embedded graph $G$ are related by reversing the cyclic orderings of the incident half-edges at some of the vertices of the combinatorial map. 

In general, a {\em local embedding} of a combinatorial map $M$ is a cellular embedding of the underlying abstract graph of $M$ into some surface such that the cyclic order at the vertices of $M$ is  preserved with respect to one or the other of the local orientation of a neighbourhood of the image of each vertex. (We use the term `local embedding' to emphasize the fact that we are considering embeddings of $M$ that do not necessarily arise from the the standard identification between combinatorial maps and  graphs embedded in an orientable surface.)

We will say that two embedded graphs $G$ and $H$ are {\em equivalent as locally embedded maps} if and only if it is possible to assign orientations to the vertices of $G$ and $H$ such that the resulting underlying combinatorial maps are equal. 

As an example, consider the combinatorial map $M$ on two vertices, with edges $e, f, g, h$,
in that cyclic order at both vertices, and the combinatorial map $N$ on two vertices with
with cyclic order $e, f, g, h$ at one vertex and cyclic order $e, f, h, g$ at the other vertex.
Then the two embedded graphs obtained as the standard orientable embeddings of $M$ and $N$ are equivalent as abstract graphs, but not as locally embedded maps. On the other hand, the two embedded graphs obtained by locally embedding $M$ on the sphere or the torus are not equivalent as embedded graphs, but are equivalent as locally embedded maps.

We note that $G$ and $H$ are equivalent as locally embedded maps if and only if they have the same set of underlying combinatorial maps, and that $M$ is an underlying combinatorial map of $G$ if and only if $G$ is a local embedding of $M$.  Thus, if $G$ is an embedded graph and $M$ is any of its underlying combinatorial maps, then $\{ \text{local embeddings of } M\} 
=\{ H \;| \;  H \text{ and } G \text{ are equivalent as locally embedded maps} \}.$

The hierarchy of these equivalences is perhaps clearest in the language of signed rotations systems.   Two graphs $G$ and $H$ represented by signed rotation systems are equivalent as abstract graphs if the abstract graphs in their signed rotation system representations are isomorphic (i.e., we `forget' both the cyclic orders and the edge signs). They are equivalent as locally embedded graphs if the abstract graphs are isomorphic and the cyclic orders are the same (i.e., we `forget' just the edge signs), remembering that cyclic orders may be reversed.  Finally, $G$ and $H$ are equivalent as embedded graphs if they have the same signed rotation systems. This hierarchy is also implicit in Gross and Tucker \cite{GT87}.

Note that, given an arrow presentation of an embedded graph $G$, one can obtain an underlying combinatorial map by forgetting the directions on the arrows and retaining only their cyclic order about the vertex disc.

\medskip

Medial graphs play a central role throughout this paper.  If $G$ is cellularly embedded, we constructed its medial graph $G_m$ exactly as in the plane case, by placing a vertex of degree 4 on each edge, and then drawing the edges of the medial graph by following the face boundaries of $G$. There is a natural embedding of the medial graph $G_m$, viewed as a ribbon graph, into $G$ viewed as a ribbon graph.  This results from drawing the edges of $G_m$ very close to the the edges of $G$ in the surface, and taking a smaller neighbourhood of $G_m$ than the neighbourhood of $G$ when cutting out the ribbon graphs from the surface. See Figure~\ref{fig.meddef}.  An example of a medial ribbon graph is given in Figure~\ref{fig.medex}. Consistent with this definition is that the medial graph of an isolated vertex is an isolated face, and we adopt this convention.

\begin{figure}
\[\includegraphics[height=2cm]{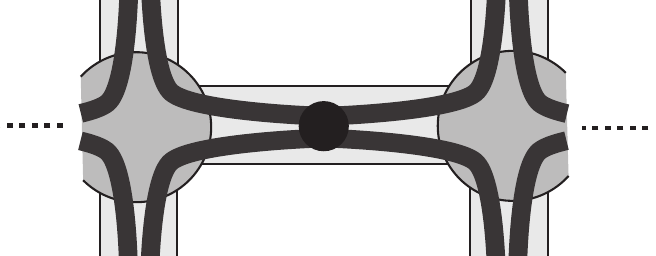}\]
\caption{The formation of a medial ribbon graph.}
\label{fig.meddef}
\end{figure}

 \begin{figure}
\begin{tabular}{p{4cm}cp{4cm}cp{4cm}}
\includegraphics[width=30mm]{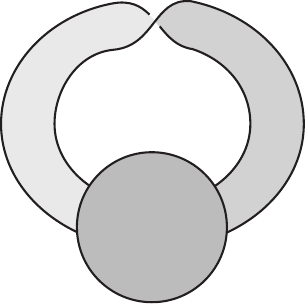} && \includegraphics[width=30mm]{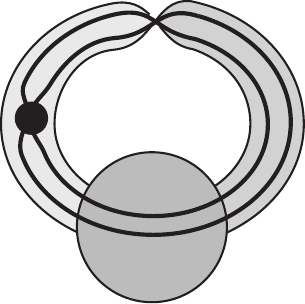} && \includegraphics[width=30mm]{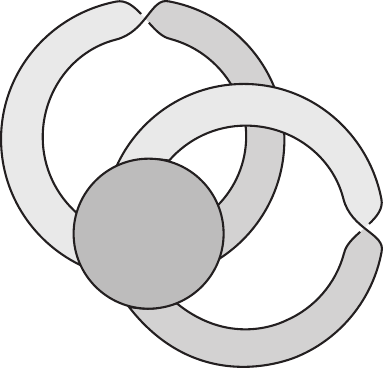}  \\
A ribbon graph $G$. && The  medial graph $G_m$ drawn inside $G$. &  & $G_m$ presented as a ribbon graph.
\end{tabular}
\caption{An example of a medial ribbon graph.}
\label{fig.medex}
\end{figure}

\section{Twisted duality and the ribbon group action}\label{s.td}

Duality in the plane is constrained by the fact that the result of taking the dual of a plane graph is again a plane graph.  
Working with embedded graphs, however, allows greater flexibility. Here we move out of the class of plane graphs by  allowing two operations on the edges of an embedded graph $G$. The first is forming the dual with respect to an individual edge, as defined by Chmutov in \cite{Ch1}, and the second is adding a half-twist to an edge. These two operations give rise to a group action of ${S_3}^{e(G)}$ on $G$, which we call the \emph{ribbon group action}. The ribbon group action is the foundation of many of the results in the later sections of this paper.

\subsection{The ribbon group action}\label{ss.rga}

As often is the case, we begin with graphs equipped with a linear ordering on their edges, and then show that our constructions are independent of these orderings.  We first give the half-twist and dual operations with respect to single distinguished edges.  We  define the \emph{ribbon group action} of $\fG^{e(G)}$ on graphs with a linear ordering on their edges, where $\fG\cong S_3$. We then provide a more efficient notation which is independent of the ordering of the edges and also establish some elementary properties of the ribbon group action. 

We let $\calG$ denote the set of embedded graphs considered up to homeomorphism (we identify an embedded graph $G$ with its homeomorphism class), and we let $\calG_{(n)} \subseteq \calG$ denote the set of embedded graphs  with exactly $n$ edges. We write \[\calG_{or} = \left\{  (G,\ell)| G \in \calG \text{ and } \ell \text{ is a linear ordering of the edges} \right\}\] for the set of embedded graphs with ordered edges, and \[\calG_{or (n)} = \left\{  (G,\ell)| G \in \calG_{(n)} \text{ and } \ell \text{ is a linear ordering of the edges} \right\}\] for those with exactly $n$ edges.

\begin{definition}\label{def.ops}

Let $(G,\ell) \in \calG_{or}$ and suppose $e_i$ is the $i^\text{th}$ edge in the ordering $\ell$. Also, suppose $G$ is given in term of its arrow presentation, so $e_i$ is a label of a pair of arrows.  

The \emph{half-twist of the $i^{\text{th}}$ edge} is $(\tau, i)(G,\ell)=(H,\ell)$ where $H$ is obtained from $G$  by reversing the direction of exactly one of the $e_i$-labelled arrows of the arrow presentation, as in Figure~\ref{taudelta}.  $H$ inherits its edge order $\ell$ in the natural way from $G$.  

The \emph{dual with respect to the $i^{\text{th}}$ edge} is $(\delta,i) (G,\ell) = (H,\ell)$, where $H$ is obtained from $G$ as follows.  Suppose $A$ and $B$ are the two arrows labelled $e_i$ in the arrow presentation of $G$.  Draw a line segment with an arrow on it directed from the the head of $A$ to the tail of $B$, and a line segment with an arrow on it directed from the head of $B$ to the tail of $A$.  Label both of these arrows $e_i$, and delete $A$ and $B$ with the arcs containing them. The line segments with their arrows become arcs of a new circle (or circles) in the arrow presentation of $H$.  As with the half-twist, $H$ inherits its edge order $\ell$ from $G$. See Figure~\ref{taudelta}.

\end{definition}

\begin{figure}
\[  \tau\left( \;\; \raisebox{-4mm}{\includegraphics[height=10mm]{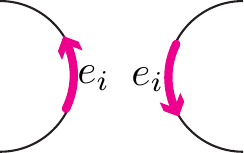}}\;\; \right) \; = \; \;\raisebox{-4mm}{\includegraphics[height=10mm]{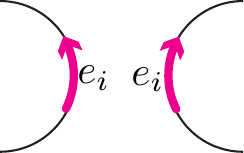}} \hspace{2cm} \delta\left( \;\; \raisebox{-4mm}{\includegraphics[height=10mm]{a1ei}} \;\; \right) \; = \; \; \raisebox{-4mm}{\includegraphics[height=10mm]{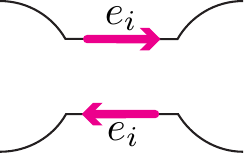}} \]
%
%
\caption{$\tau$ and $\delta$ with arrow presentations. }
\label{taudelta}
\end{figure}

Note that the half-twist operation does not change the number of vertices nor the cyclic order of incident half-edges at a vertex in a graph, but the dual operation may change either.

We make the simple but important observation that the half-twist and dual operations when applied to \emph{different} edges commute.

\begin{proposition}\label{switch}

If $i \neq j$ and $\g, \h \in \{ \tau, \delta \}$, then $(\g,i) ((\h,j)(G, \ell)) = (\h,j)((\g,i)(G,\ell))$.

\end{proposition}

However,  $(\tau, i)$ and $(\delta,i)$ do {\em not} commute when applied to the same edge. In fact, we will see they induce a group action of $S_3$ on that edge.

We use the following notation to denote compositions applied to the same edge:  \[(\g\h, i)(G,\ell):=(\g,i)((\h,i)(G,\ell)),\] where $\g, \h \in \{ \tau, \delta \}$.  We also define $(1,i)(G,\ell):=(G,\ell)$.  Thus, we can consider the action of $(\g,i)$ on $(G,\ell)$ where $\g$ is a word in $\{ \tau, \delta \}$.  

\begin{lemma}\label{l.action1}
If $(G,\ell) \in \calG_{or}$ then, for each fixed $i$, \[{(\tau^2, i)}(G,\ell)={ (\delta^2,i)}(G,\ell)=  ({(\tau \delta)}^3, i)(G,\ell)=1(G,\ell).\]  Therefore, given a fixed $i$, there is an action of the symmetric group $S_3$, with the presentation \[S_3 \cong \fG := \langle  \delta, \tau   \; |\;   \delta^2, \tau^2, (\tau \delta)^3   \rangle ,\] on $\calG_{or}$.

\end{lemma}
\begin{proof}
The following calculations verify that ${ (\delta^2,i)}(G,\ell)= {(\tau^2, i)}(G,\ell)= ((\tau \delta)^3,i)(G,\ell)=1(G,\ell)$.
Since \[    \raisebox{-4mm}{\includegraphics[height=10mm]{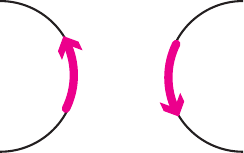}} \;\; \overset{(\tau, i)}{\longmapsto} \;\; 
 \raisebox{-4mm}{\includegraphics[height=10mm]{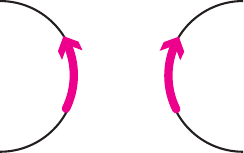}} \;\; \overset{(\tau, i)}{\longmapsto} \;\; 
  \raisebox{-4mm}{\includegraphics[height=10mm]{a1}},
 \]
 then ${(\tau^2, i)}=1$. Also 
\[    \raisebox{-4mm}{\includegraphics[height=10mm]{a1}} \;\; \overset{(\delta,i)}{\longmapsto} \;\; 
 \raisebox{-4mm}{\includegraphics[height=10mm]{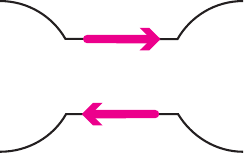}} \;\; \overset{(\delta,i)}{\longmapsto} \;\; 
  \raisebox{-4mm}{\includegraphics[height=10mm]{a1}},
 \]
 giving the identity ${(\delta^2,i)}=1$. Finally,
\begin{multline*}    \raisebox{-4mm}{\includegraphics[height=10mm]{a1}} \;\; \overset{(\tau, i)}{\longmapsto} \;\; 
 \raisebox{-4mm}{\includegraphics[height=10mm]{a2}}  \;\; \overset{(\delta,i)}{\longmapsto} \;\; 
  \raisebox{-4mm}{\includegraphics[height=10mm]{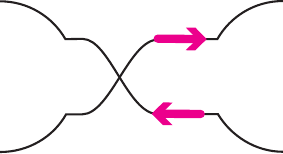}}  \;\; \overset{(\tau, i)}{\longmapsto} \;\; 
   \raisebox{-4mm}{\includegraphics[height=10mm]{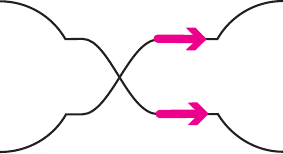}}  \;\; 
   \\  \\\overset{(\delta,i)}{\longmapsto} \;\; 
     \raisebox{-4mm}{\includegraphics[height=10mm]{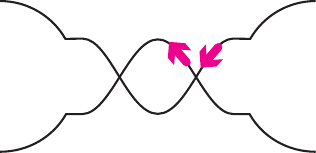}}  \;\; = \;\; 
     \raisebox{-4mm}{\includegraphics[height=10mm]{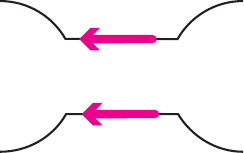}}  \;\; \overset{(\tau, i)}{\longmapsto} \;\; 
      \raisebox{-4mm}{\includegraphics[height=10mm]{a3}}  \;\; \overset{(\delta,i)}{\longmapsto} \;\; 
 \raisebox{-4mm}{\includegraphics[height=10mm]{a1}},  \end{multline*}
and thus $(({\tau \delta})^3, i)=1$.

We have shown that $\delta$ and $\tau$ satisfy the defining relations of the presentation.  It remains to show that  $\delta$ and $\tau$ satisfy no additional relations. Since every quotient group of $S_3$ is abelian, it is enough to show that $\tau \delta  (G,\ell) \neq \delta \tau (G, \ell) $ for some $(G,\ell)$. This fact is readily verified.  
\end{proof}

The action in Lemma~\ref{l.action1} for fixed $i$ now extends to a group action of ${S_3}^{e(G)}$ on $\calG_{or(n)}$.

\begin{definition}\label{ribbon group}
We call ${S_3}^{n} \cong \fG^{n}$ the \emph{ribbon group for $n$ edges} and define the \emph{ribbon group action} of the ribbon group on $\calG_{or(n)}$ by:
\begin{eqnarray*} ( \g_1, \g_2,\g_3, \ldots , \g_n )(G,\ell) &= (\g_n,n)((\g_{n-1}, n-1) \ldots ((\g_2,2)(( \g_1,1)(G,\ell))) \ldots) \\
&=  ( (\g_n,n)\circ(\g_{n-1}, n-1)\circ \ldots \circ (\g_2,2) \circ ( \g_1,1)   ) (G,\ell) 
,\end{eqnarray*}
where  $\g_i \in \fG$ for all $i$.
\end{definition}

With Definition~\ref{ribbon group}, if $(G,\ell) \in \calG_{or(n)}$, then we can view $(\tau,i)$ as an element of $\fG^{n}$ of the form $(1, \ldots, \tau, \ldots, 1)$, with $\tau$ as the $i^{\text{th}}$ coordinate, and similarly for $\delta$ and the other elements of $\fG$.

\subsection{Twisted duals}\label{TwDu}

We now define twisted duality for graphs without any edge ordering.  Our final definition of a twisted dual of $G$ will be of the form
 \[
G^{\prod{\g_i(A_i)}},
\] 
\noindent
where the $A_i$'s  partition  the edge set, and the $\g_i$'s are in $\fG$. However, we need some preliminary definitions to make sense of this expression. 

We begin with an obvious proposition.

\begin{proposition} \label{permute}

If $(G,\ell) \in \calG_{or(n)}$, $\vh \in {\fG}^n$, and $\sigma \in S_n$, then $\vh(G,\ell) = \sigma(\vh) (G, \sigma(\ell))$, where $\sigma$ acts on $\vh$ by permuting the order of the elements of the $n$-tuple. 

\end{proposition}

Proposition~\ref{permute} ensures that $G^{\prod{\g(A_i)}}$, as given below, is well-defined, that is, it is independent of the edge ordering $\ell$.

\begin{definition}\label{unordered twistdual}

Suppose $G \in \calG_{(n)}$, $A, B \subseteq E(G)$ and  $\g, \h \in \fG$.  Define $G^{\g(A)}$ as follows. Let $\ell$ be an arbitrary ordering $(e_1, \ldots, e_n)$ of the edges of $G$, and define $\vg_{A} :=(\epsilon_1, \ldots, \epsilon_n) \in \fG^n$, where $\epsilon_i = \g$ if $e_i \in A$ and $\epsilon_i=1$ else. Then, 
\[
G^{\g(A)}:=\vg_{A}(G,\ell).
\] 
Moreover, we establish the following notational conventions:
\[
G^{\g(A)\h(B)}:=(G^{\g(A)})^{\h(B)} \text{, and}  \quad
G^{\g\h(A)}:=G^{\h(A)\g(A)}.
\]

\end{definition}

\begin{proposition}\label{reduce} 
If, for all $i$, we have $\h_i \in \fG$ and $B_i \subseteq E(G)$, then any expression of the form  
$
G^{\prod {\h_i(B_i)}}
$
 is equal to 
$
G^{\prod^6_{i=1}{\g_i(A_i)}}
$, 
where the $A_i\subseteq E(G)$ are pairwise disjoint with $\cup_i {A_i}=E(G)$, and where $\g_1=1, \g_2=\tau , \g_3=\delta , \g_4=\tau \delta , \g_5= \delta \tau$, and  $\g_6 = \tau \delta \tau \in \fG$.  Here, the terms in the product  $\prod {\h_i(B_i)}$ do not necessarily commute, while the terms in the product  $\prod^6_{i=1}{\g_i(A_i)}$ do commute with each other.

\end{proposition} 

Proposition~\ref{reduce} follows from repeated applications of Definition~\ref{unordered twistdual}, and the fact that   the terms in the product $\prod^6_{i=1}{\g_i(A_i)}$ commute follows from the fact that the $A_i$'s are disjoint and Proposition~\ref{switch}.  Hereafter, we will customarily omit any factors of the form $1(A_1)$ or $\g_i(\emptyset)$ in these expressions. Also, if $A_i$ is given explicitly by a list of edges, to simplify notation, we will omit the set brackets, for example, writing $\tau(e,f)$ for $\tau(\{e,f\})$.  

As an example of Proposition~\ref{reduce}, if $G$ is an embedded graph with edges $d, e, f, g, h$, then
\[
G^{\tau(d,e,f) \delta (e,f,g)} = (G^{\tau(d,e,f)})^{\delta (e,f,g)} = G^{\tau(d) (\tau(e,f) \delta(e,f)) (\delta(g))} = G^{\tau(d) \delta \tau(e,f) (\delta(g))}= G^{\tau(d)\delta(g) \delta \tau(e,f)}.
\] 
The edge $h$ is unaffected.

Geometrically, these maps act on an edge $e$ of a ribbon graph in the following way: $\tau$ adds a half-twist to the edge $e$, and $\delta$ forms the partial dual  at the edge $e$. Products of $\tau$ and $\delta$ are applied to the edge successively.  An illustration of the actions of $\tau$ and $\delta$ on an embedded  graph is given in Example \ref{e.pd}.
 \begin{example}\label{e.pd}
 If $G$ is an embedded graph with $E(G)=\{e_1,e_2\}$, with the order $(e_1,e_2)$, represented as an arrow presentation and as a ribbon graph shown below,
\[
\raisebox{8mm}{$G =$}   \;\;\raisebox{1mm}{\includegraphics[width=30mm]{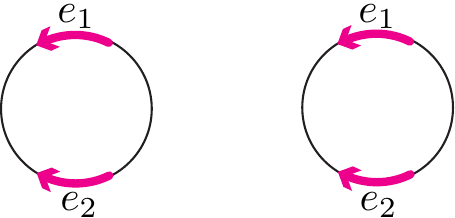}} \; \;\raisebox{8mm}{$=$} \;  \; \raisebox{0mm}{\includegraphics[width=34mm]{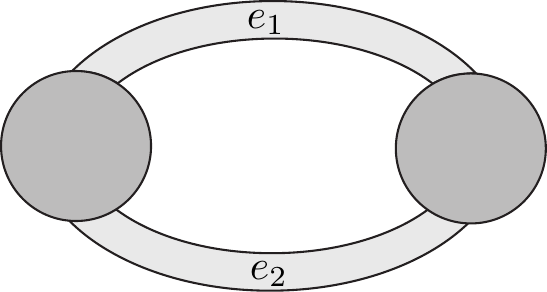}} \;\raisebox{8mm}{,}
\]
then we have
\[
\raisebox{8mm}{$(\tau, 1)(G) =G^{\tau(e_1)} =$}   \;\;\raisebox{1mm}{\includegraphics[width=30mm]{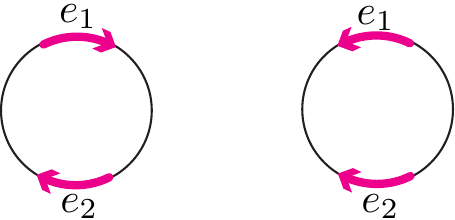}} \; \;\raisebox{8mm}{$=$} \;  \; \raisebox{0mm}{\includegraphics[width=34mm]{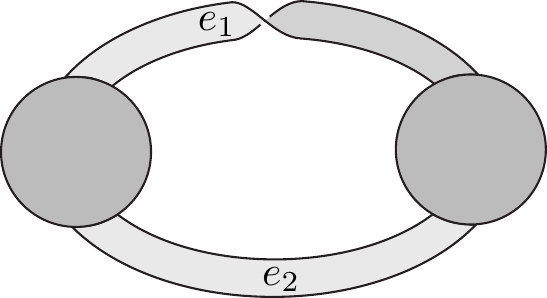}} \;\raisebox{8mm}{,}
\]
and 
\[
\raisebox{8mm}{$(\delta , 1)(G) =G^{\delta(e_1)} =$}   \;\;\raisebox{4mm}{\includegraphics[width=30mm]{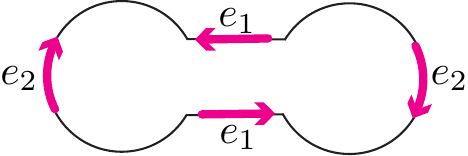}} \; \;\raisebox{8mm}{$=$} \;  \; \raisebox{1mm}{\includegraphics[width=24mm]{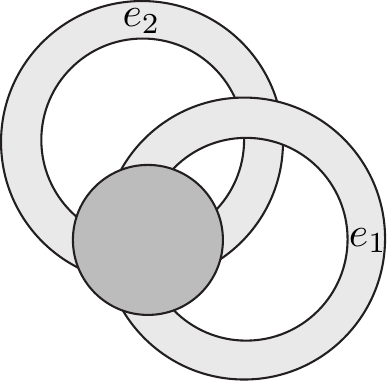}} \;\raisebox{8mm}{.}
\]

The full orbit of $G$ is given in Figure~\ref{Fig:orbit}.

\begin{figure}
\[\includegraphics[height=4cm]{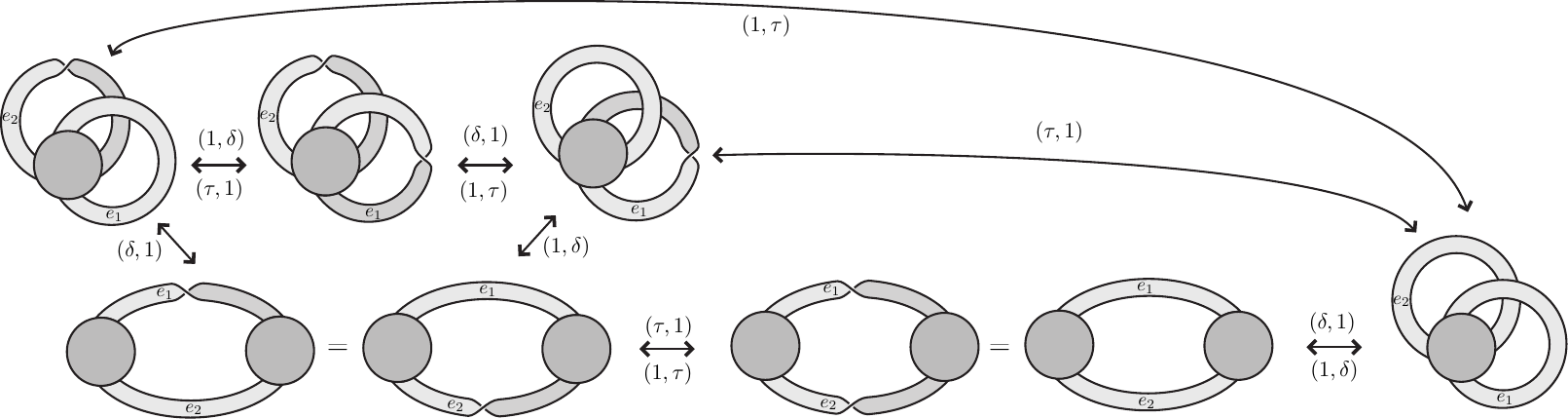}\]
\caption{The orbit of a plane digon under the ribbon group action.}
\label{Fig:orbit}
\end{figure}
\end{example}

\bigskip

We now have the following definition of a \emph{twisted dual of an embedded graph}.

\begin{definition}\label{geo twistdual}
If $G$ is an embedded graph, then $H$ is a \emph{twisted dual} of $G$ if it can be written in the form
\[
H=G^{\prod^6_{i=1}{\g_i(A_i)}},
\] 
  
\noindent
where the $A_i$'s partition $E(G)$, and the $\g_i$'s are the six elements of $\fG$. 

Equivalently, if $\ell$ is an arbitrary ordering of the edges of $G$, then $H$ is a \emph{twisted dual} of $G$ if $\vg (G,\ell)= (H,\ell)$ for some $\vg \in \fG^{e(G)}$, \emph{i.e.} $H$ is a twisted dual of $G$ if $(H,\ell)$ is in the orbit of $(G,\ell)$ under the ribbon group action. Note that twisted duality is symmetric, so we may speak of $H$ and $G$ as being twisted duals of one another.
\end{definition}

The linear ordering of the edges is necessary to define a group action,
 but is not necessary for the construction of twisted duals.  Thus we will use Definition~\ref{ribbon group} when we wish to emphasize the group action, and Definition~\ref{geo twistdual} to emphasize the geometry. However, since moving between the two viewpoints is so natural, our language may not always preserve the distinction, as for example, we may speak of $\fG^n$ ``acting'' on an unordered graph.

One of our main interests in this paper is the orbits $Orb(G,\ell):=\fG^{e(G)} (G, \ell)=\{ \vg(G, \ell) \; | \; \vg \in\fG^{e(G)}  \}$ of the group action.  With slight abuse of terminology, we define \[Orb(G):=\{H : (H, \ell) \in Orb(G, \ell) \text{ for some edge order } \ell \}.\]

\begin{remark}
Although one of our primary objectives in this paper is investigating and characterizing various orbits of the ribbon group action, the stabilizer subgroups are also of interest.  In particular, the stabilizer subgroup of $\fG^{e(G)}$, like the automorphism group of a graph, is an invariant of $G$. The study of stabilizers also encompasses the study of self-dual and self-Petrial graphs.  Although we do not investigate the stabilizers here, they certainly  warrant study. Another interesting question, again outside the scope of this paper, is how ribbon graph theoretic properties, such as the genus or number of vertices, vary over the elements in an orbit. 
\end{remark}

\subsection{Some properties of the ribbon group action}

In general, the group action $\fG^n$ on $\calG_{or(n)}$ is faithful and transitive, but is not  
free and has no fixed points. However, it turns out that stronger,  
order-independent analogues of these properties hold for twisted  
duality. We will state and prove these stronger analogous properties in  
Proposition~\ref{p.action} and then state and prove the ordered versions  in the language of group actions
in Corollary~\ref{c.action}.

\begin{proposition} \label{p.action}
~

\begin{enumerate}
\item Let $\vg \in \fG^n$, then  $\vg (G, \ell)  \in \{(G, \pi (\ell) ) 
\;|\; \pi \in \fS_n   \} $ for all $G\in \calG_n$ if and only if $\vg=\bf{1}$.
\item For all $G\in \mathcal{G}_n$ there is some  $\vg \in \fG^n$  such that  
$\vg (G, \ell)  \notin \{(G, \pi (\ell)) \;|\; \pi \in \fS_n \} $. \item For all $G\in \mathcal{G}_n$, there exists an $G' \in \mathcal{G}_n$ such  
that $G'$ is not a twisted dual of $G$, if and only if $n>1$

\end{enumerate}
\end{proposition}
\begin{proof} ~

\begin{enumerate}
\item
The sufficiency is easily verified by calculation.

To prove the necessity, for each $\vg \neq \bf{1}$ we find $ 
(G,\ell) \in \mathcal{G}_{or(n)}$ with the property that, if $\vg (G,\ell) = (H,\ell)$, then $G \neq H$ as embedded graphs, and hence $(H,\ell) \neq (G, \pi (\ell))$ for any ordering $\pi (\ell)$.

Given $\vg \neq \bf{1}$, so that $\vg=(\g_1, \g_2, \ldots , \g_n)$ has $\g_i \neq 1$ for some $i$, we consider $(B, \ell)$ and $ (B', \ell) \in \mathcal{G}_{or(n)}$, where $B$ is the connected plane bouquet, which consists of $n$ loops at a single vertex, and $(\tau, i)(B,\ell) = (B', \ell)$.

If $\g_i= \delta$ or $\tau \delta$, then $\vg (B, \ell)$ has at least one more  
 vertex than $(B, \ell)$ (since $(\g_i,i)(B,\ell)$ has two vertices), so $\g (B, \ell) \neq (B, \pi(\ell))$ for any permutation $\pi$.
If $\g_i=\tau$ then $\vg (B, \ell)$ is non-orientable (as $(\g_i,i)(B, 
\ell)$ is non-orientable), so again $\g (B, \ell) \neq (B, \pi(\ell))$ for any permutation $\pi$.
If $\g_i= \delta\tau$ or $\tau \delta\tau$, then $\vg (B', \ell)$  
has at least one more vertex than $(B', \ell)$ (since $(\g_i,i)(B',\ell)$ has two  
vertices). In all of these cases the $\vg$ changes  
the underlying embedded graph as required.

\item Let $(G,\ell) \in \calG_{or(n)}$, then $G$ either  
contains a cycle or does not contain a cycle. If $G$ contains  a  
cycle then there exists a set of edges $A$ of $G$ such that adding a  
half-twist to each of the edges in $A$ will change the orientability  
of $G$.  Let $\vg =(\g_1, \ldots, \g_n)$, where $\g_i=\tau$ if $e_i \in A$ and $0$ otherwise. It then follows that exactly one of  $(G,\ell)$ and $\vg (G,\ell)$ is orientable.

  On the other hand, if $G$ does not contain a cycle,  then taking  
the dual at any edge $e$ of $G$  will result in  a  
ribbon graph containing a cycle.  In this case, if we let $\vg= 
( (\delta , 1 ), (1,2), \ldots (1,n)  )$, we have that exactly one  
of $(G,\ell)$ and $\vg (G,\ell)$ contains a cycle.

  In either case there exists some $\vg\in \fG^n$ with the property  
that the ribbon graphs (without an edge order) in $(G,\ell)$ and $\vg  
(G, \ell)$ are distinct, and therefore $(G,\ell)\neq \vg (G, \pi(\ell))$,  
for any edge order $\pi(\ell)$.

\item Sufficiency is easily checked by calculation. To prove the  
necessity,  assume that $n>1$. We need to show that for any graph $G \in \mathcal{G}_n$ there is some graph $H \in \mathcal{G}_n$ which is not a twisted dual of $G$.

We will prove the result by showing that there exists a set $\mathcal{S}$ of ribbon  
graphs in $\mathcal{G}_n$ that is closed under taking the twisted  
duals and that has the additional property that every orientable ribbon graph in the set  
is plane. Thus, any graph $G \in \mathcal{G}_n - \mathcal{S}$ is not a twisted dual of any graph in $\mathcal{S}$ and vice versa.  However, for all $n>1$, there is an orientable      non-plane ribbon graph, so $\mathcal{G}_n - \mathcal{S}$ is not empty.  

Our desired set $\mathcal{S}$ has 
the property that $G \in \mathcal{S}$ if and only if   
$G$ is plane and has a distinguished vertex $v$ such that  every  
edge in $G$ is either a loop incident with $v$ or the edge is a  
bridge incident to $v$ and a $1$-valent vertex; or $G$ is  obtained by half-twisting some of the edges of  
such a ribbon graph. Notice that every orientable ribbon graph in $ 
\mathcal{S}$ is plane.

To show that $\mathcal{S}$ is closed under the operations of twisted  
duality it is enough to show that for each $G\in \mathcal{S}$, $e\in E 
(G)$ and $\g\in \fG$, we have $G^{\g(e)}\in \mathcal{S}$. To see that  
this is indeed the case, let $e$ be an edge of some $G\in \mathcal{S} 
$. Then $e$ is either a non-twisted loop, a twisted loop, or a bridge,  
and $\g =1,\tau, \delta, \tau\delta, \delta\tau$ or $\tau\delta\tau$.
   If $e$ is a non-twisted loop then the edge corresponding to $e$ in  
$G^{1(e)}$ and $G^{\tau\delta\tau (e)}$ is a non-twisted loop;  in $G^ 
{\delta(e)}$ and $G^{\tau\delta(e)}$ it is a bridge; and in $G^{\tau(e)} 
$ and $G^{\delta\tau (e)}$ it is a twisted loop.
   If $e$ is a twisted loop then the edge corresponding to $e$ in $G^ 
{\tau(e)}$ and $G^{\tau\delta (e)}$ is a non-twisted loop;  in $G^ 
{\delta\tau(e)}$ and $G^{\tau\delta\tau(e)}$ is a bridge; and in $G^{1 
(e)}$ and $G^{\delta (e)}$  is a twisted loop.
   If $e$ is a bridge then the edge corresponding to $e$ in $G^{\delta 
(e)}$ and $G^{\delta\tau (e)}$ is a non-twisted loop;  in $G^{\delta 
\tau(e)}$ and $G^{\tau\delta\tau(e)}$ is a bridge; and in $G^{1(e)}$  
and $G^{\delta (e)}$  is a twisted loop.
   In all of these cases the resulting twisted dual $G^{\g(e)}$ is  
in $\mathcal{S}$ and the result then follows.

\end{enumerate}

\end{proof}

\begin{corollary}\label{c.action}
The action of $ \fG^{n}$ on $\calG_{or(n)}$ is \begin{enumerate}
\item  faithful;
\item  has no fixed points;
\item transitive if and only if $n>1$;
\item not free.
\end{enumerate}
\end{corollary}

\begin{proof}
The first three properties follow easily from Proposition~\ref{p.action}.

To show that the action is not free we need to show that there is a $ 
(G,\ell)\in \calG_n$,  such that $\g (G,\ell)=(G,\ell)$ for some non-trivial $\vg\in \fG^n$. This exists since $\prod_{i=1}^n (\tau, i)(K_{1,n}, \ell)  
= (K_{1,n}, \ell)$, for any edge order $\ell$ of the complete  
bipartite ribbon graph $K_{1,n}$.
\end{proof}

\subsection{Some historical context}\label{ss.notable}

We take a moment now to place prior work on partial duals and Petrie duals (also known as Petrials) in the context of twisted duality.  Doing so will also highlight the actions of some important subgroups of the ribbon group. In particular, we will see that sets of geometrically dual graphs arise as the orbits of the action of an order two subgroup of $\fG^{e(G)}$, and that sets of partially dual graphs, which result from applying $\delta$ to a subset of edges, arise as orbits of the action of another natural subgroup of the ribbon group.   The half-twist operation $\tau$ gives an analogous construction for Petrials and what we will term \emph{partial Petrials}.  We will use Wilson's taxonomy from \cite{Wil79} for similar constructions for each of the subgroups of $\fG$. 

The group $\fG=\langle \delta, \tau \; |\; \delta^2, \tau^2 , (\delta\tau)^3 \rangle$ has five non-trivial subgroups, with the four proper, non-trivial, subgroups being cyclically generated. Each of these  subgroups defines an action on the sets   $\calG_{or(n)}$.
If $\g\in \fG$ and $\langle \g \rangle$ is the subgroup of $\fG$ generated by $\g$, then we denote the orbit of $(G,\ell)$ under the action of  $\langle\g \rangle^{e(G)} $  on $\calG_{or(n)}$ by
$Orb_{( \g )}(G,\ell):=\{ \vg(G, \ell) \; | \; \vg \in \langle\g \rangle^{e(G)} \, \leq \,  \fG^{e(G)}  \}$.  Again with slight abuse of terminology, we define
 \[Orb_{( \g )}(G):=\{H : (H, \ell) \in Orb_{( \g )}(G, \ell) \text{ for some edge order } \ell \}.\]

The following fact relating twisted duality and geometric duality will be important later.  Recall that the geometric dual $G^*$ of an cellularly embedded graph $G$ is constructed exactly as in the plane case by placing a vertex in each face, and connecting two of these vertices with an edge whenever their faces share an edge on their boundaries.  In the context of ribbon graphs, $G^*$ is constructed by  regarding the ribbon graph $G$ as a punctured surface, filling in the punctures using a set of discs denoted $ \V(G^*) $, then removing the original vertex set $\V(G)$ (so $G^* = (\V(G^*), \E(G)) $).

\begin{proposition}[Chmutov \cite{Ch1}]
If $G$ is an embedded graph, then
\[ 
G^* = G^{\delta(E(G))}. 
\] 
\end{proposition}

Thus, we see that a graph and its geometric dual form the orbit of $G$ under the action of the subgroup of $\fG^{e(G)}$ of order two generated by $(\delta, \ldots, \delta)$.

\begin{definition}
Let $  \langle \delta \rangle = \langle \delta \; |\; \delta^2\rangle$  be the subgroup of $\fG$ generated by $\delta$. Then two twisted duals $H$ and $G$  are said to be {\em partial duals} if $(H,\ell)= \vg(G,\ell)$ for some $\vg\in  \langle \delta \rangle^{e(G)}$ and some edge ordering $\ell$, \emph{i.e.} if $H \in Orb_{( \delta )}(G)$.
\end{definition}

Chmutov further observed in \cite{Ch1} that partial duality can be regarded as an action of  $\mathbb{Z}_2^{e(G)}$ on a ribbon graph $G$, and here we see that  this $\mathbb{Z}_2^{e(G)}$ action is the action of a subgroup of the ribbon group $\fG^{e(G)}$.

 \medskip

A very similar situation occurs with Petrials.  The Petrial of an embedded graph $G$, which we denote by $G^{\times}$, is formed with the same edges and vertices as $G$, but for the faces taking the Petrie polygons, which are the result of closed left-right walks in $G$ (see Wilson \cite{Wil79}).  Since the left-right pattern effectively means crossing over each edge, when the graph is viewed as a ribbon graph, this is simply giving each edge a half-twist, and hence $G^{\times}$ is simply the result of giving a half-twist to all of the edges.  Thus, we have the following proposition, and see that, in analogy with geometric duality, a graph and its Petrial also form  the orbit of $G$ under the action of the subgroup of $\fG^{e(G)}$ of order two generated by $(\tau, \ldots, \tau)$. 

\begin{proposition}
If $G$ is an embedded graph, then
\[ 
G^{\times} = G^{\tau(E(G))}. 
\] 
\end{proposition}

This parallel suggests naming the orbit under half-twists the \emph{partial Petrials}, in analogy with the partial duals given above.

\begin{definition}
Let $  \langle \tau \rangle = \langle \tau \; |\; \tau^2\rangle$  be the subgroup of $\fG$ generated by $\tau$. Then two twisted duals $H$ and $G$  are said to be {\em partial Petrials} if $(H,\ell)= \vg(G,\ell)$ for some $\vg\in  \langle \tau \rangle^{e(G)}$ and some edge ordering $\ell$, \emph{i.e.} if $H \in Orb_{( \tau )}(G)$.
\end{definition}

To emphasize the topology, we will sometimes say that two partial Petrials are \emph{twists} of one another.

We will now see that the orbit of $G$ under the action of the subgroup $\langle \tau \rangle^{e(G)}$ completely describes equivalence as local embeddings. 

\begin{proposition}\label{part comb}
Let $G$ be an embedded graph and $M$ be any one of its underlying combinatorial maps.  Then 
\begin{multline*}Orb_{( \tau )}(G) = \{  G^{\tau(A)}\; |\; A\subseteq E(G) \} 
= \{ \text{local embeddings of } M\} \\
=\{ H \;| \;  H \text{ and } G \text{ are equivalent as locally embedded maps} \}
. \end{multline*}
In particular, two graphs $G$ and $H$ are partial Petrials if and only if they are equivalent as locally embedded maps.
\end{proposition}
\begin{proof}
The final equality was noted in Section \ref{Sec Embedded graphs}.  
Furthermore, $H$ and $G$ are equivalent as locally embedded maps if and only if there exist local orientations of $H$ and $G$ that give rise to the same combinatorial map if and only if the signed rotation systems of $H$ and $G$ with respect to these local orientations differ by signs if and only if $G$ and $H$ are partial Petrials.

\end{proof}

In \cite{Wil79}, Wilson found that the operations of taking the geometric dual and the Petrial of a graph generated an action of $S_3$ on the graph.  This action is exactly the action of the copy of $S_3$ appearing as the subgroup of $\fG^{e(G)}$ generated by $( \delta, \ldots, \delta)$ and $( \tau, \ldots, \tau)$.  In addition to the geometric dual and the Petrial already described, Wilson named these global operations, as given in Table \ref{tax}.  We suggest the name Wilsonial in parallel with Petrial for the third operation that is a `duality' property, instead of Wilson's original term `opposite'.  Also, in order to emphasize the operations and their geometry, we use the term twisted dual here  rather than Wilson's original term of direct derivative.   

Just as the geometric dual and Petrial may be associated with partial duality and partial Petrials by applying $\delta$ or $\tau$ to subsets of the edges, natural classes of graphs arise as the orbits $Orb_{( \g )}(G)$ for $\g$ some other element of $\fG$. This correspondence is given in Table \ref{tax}.

\begin{table}
\begin{tabular}{|l|l|l|l|}
\hline
Generator(s) & Order of subgroup  & Applied to all edges & Applied to a subset of edges  \\
&  of $\fG$ generated   &&  \\
\hline \hline 
$\delta$ & 2&geometric dual   & partial dual \\
\hline 
$\tau$&  2 & Petrie dual  or Petrial   & partial Petrial 
\\
\hline
$\tau\delta\tau$& 2 &Wilson dual or Wilsonial & partial Wilsonial \\ & &(or the opposite)  & \\ \hline 
$\delta\tau$& 3 &triality & partial triality	
\\
\hline 
$\delta$ and $\tau$  & 6   &twisted dual (or a direct derivative)  & twisted dual \\

\hline
\end{tabular}
\bigskip ~
\caption{Taxonomy of classes of twisted duals.} \label{tax} \end{table}

\section{ Medial graphs and the ribbon group action}\label{s.medialandribbongroup}

Via the ribbon group action, twisted duality gives the full story of the interplay among a graph, its medial graph, and its various twisted duals.  This allows us to answer the question we originally posed: 
\begin{itemize}\item given any $4$-regular graph $F$, what precisely is the set of embedded graphs that have medial graphs isomorphic to $F$ as abstract graphs? \end{itemize}     We are able to provide a classification of all the twisted duals of an embedded graph $G$ via its medial graph, and thus characterize $Orb(G)$ as the answer to this question. These relations among twisted duals, embedded medial  graphs, and cycle family graphs are higher genus generalizations of the classic results relating geometric duals, medial graphs, and Tait graphs.

By way of motivation, we begin  by reviewing some basic properties of plane medial graphs. 
If $F$ is a connected $4$-regular plane graph, then $F$ is face two-colourable (see  Fleischner~\cite{Fle90a} for example). We call this a \emph{checkerboard colouring}, and use the colours black and white. The \emph{blackface graph}, $F_{bl}$, of $F$ is the plane graph constructed  by  placing one vertex in each black face and adding an edge between two of these vertices whenever the corresponding regions meet at a vertex of $F$.  The \emph{whiteface graph}, $F_{wh}$, is constructed analogously by placing vertices in the white faces.  Borrowing terminology from knot theory, we refer to $F_{bl}$ and $F_{wh}$ as the \emph{Tait graphs} of $F$.

There are two key properties of Tait graphs.
The first property is duality: $(F_{bl})^* = F_{wh}$ and $(F_{wh})^* = F_{bl}$ , where the asterisk indicates geometric duality.  Thus, if $G$ is any plane graph, and we give $G_m$ the \emph{canonical checkerboard colouring}, {\em i.e.} where the black faces contain the vertices of $G$, then 
\begin{equation}\label{taitemedial}
(G_m)_{bl}=G, \text{ and } (G_m)_{wh}=G^*.  
\end{equation}
Secondly, the medial graph of a Tait graph is just the original graph, {\em i.e.}
\begin{equation}\label{medialtaite}
(F_{wh})_m=(F_{bl})_m=F.  
\end{equation}
Moreover, $\{F_{wh}, F_{bl}\}$ is exactly the set of plane graphs whose medial graph is $F$.

With this, we can think of the Tait graphs loosely as ``orbits'' of size two under the operation of geometric duality, and everything in this ``orbit'' shares the same medial graph.  In Section~\ref{medial and cycle} we will see how this point of view can be fully realized by embedded graphs.

In the special case that a $4$-regular embedded graph $F$ (thought of as being cellularly embedded in a surface) is checkerboard colourable, then we can construct the Tait graphs just as in the plane case, and the same properties described above will still hold. In particular, there are the following two well-known results.

\begin{proposition}\label{cyclefamilydual}
If $G$ is any embedded graph, thought of as being cellularly embedded in a surface, and we canonically checkerboard colour the embedded medial graph $G_m$, then $(G_m)_{bl} = G$ and $(G_m)_{wh} = G^* = G^{\delta(E(G))}$, the geometric dual of $G$ in the surface. 
\end{proposition}

\begin{proposition}\label{blacktomedial}
Let $F$ be a $4$-regular, cellularly embedded graph. Then: 
\begin{enumerate}
\item if $F$ is checkerboard colourable, then $\{F_{bl}, F_{wh} \}$ is the complete set of cellularly embedded graphs with embedded medial graph equivalent to $F$; 
\item if $F$ is not checkerboard colourable, then $F$ is not the embedded medial graph of any embedded graph.
\end{enumerate}

\end{proposition}

Propositions \ref{cyclefamilydual} and \ref{blacktomedial}  provide a complete characterization of the embedded graphs whose medial graph is equivalent to a given $4$-regular, embedded graph $F$. Moreover, the propositions tell us that all of the embedded graphs with this property are geometric duals.

A natural question then arises: given any $4$-regular graph $F$, which embedded graphs have medial graphs isomorphic to $F$ as abstract graphs? In this section we answer this question, giving a complete characterization of the set of embedded graphs with this property.
To do this,  we introduce the concept of the cycle family graphs of  a $4$-regular,  embedded graph $F$. The cycle family graphs do not rely on the checkerboard colourability of an embedded graph, and every $4$-regular, embedded graph will admit a set of cycle family graphs. We will prove that: 
\begin{enumerate}
\item $G_m \cong H_m $ as abstract graphs if and only if $G$ and $H$ are twisted duals (compare Equation~\eqref{taitemedial} and Proposition~\ref{cyclefamilydual} to Theorem~\ref{t.radmed}); 
\item $G_m \cong F$ as abstract graphs for a given $4$-regular, embedded graph $F$ if and only if $G$ is a cycle family graph of $F$ (compare Equation~\eqref{medialtaite} and Proposition~\ref{blacktomedial} to Theorem~\ref{iso to F}).

\end{enumerate}
Thus the relationships between cycle family graphs and twisted duals fully extend the classic relations between Tait graphs and duality.

\subsection{Cycle family graphs}
In this subsection we introduce the concept of a cycle family graph. Cycle family graphs can be viewed  as an extension of Tait graphs to arbitrary ({\em i.e.} not necessarily checkerboard colourable) $4$-regular graphs.  

In order to motivate   cycle family graphs, we begin by recalling the construction of the two Tait graphs.  Suppose that $F$ is a checkerboard coloured, $4$-regular embedded graph. Then an arrow presentation of the whiteface graph, $F_{wh}$, is obtained by replacing every checkerboard coloured vertex $v$,  \raisebox{-4mm}{\includegraphics[height=10mm]{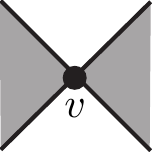}}, with the configuration   \raisebox{-4mm}{ \includegraphics[height=10mm]{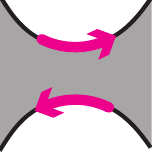}}.
   Similarly,  the blackface graph $F_{bl}$ is obtained by replacing every checkerboard coloured vertex $v$,  \raisebox{-4mm}{\includegraphics[height=10mm]{m1v}}, with the configuration  \raisebox{-4mm}{\includegraphics[height=10mm]{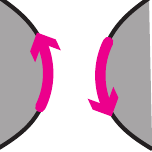}}. (The arrows in the configurations shown are labelled $v$.)
   The resulting embedded graphs $F_{wh}$ and $F_{bl}$ are the two Tait graphs of $F$.
   
  Notice that in the construction of a Tait graph, the following three restrictions are used:
  \begin{enumerate}
  \item \label{mot1}two types of splits  $\left( \raisebox{-4mm}{\includegraphics[width=1cm]{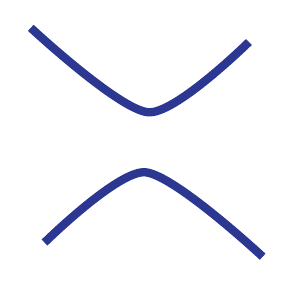}} \text{ and } \raisebox{-4mm}{\includegraphics[width=1cm]{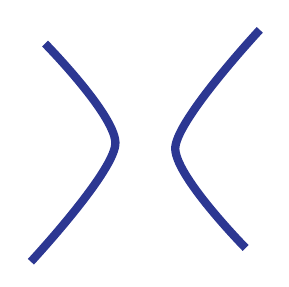}}\right)$ are used in the construction;
  \item \label{mot2} the same type of split, with respect the the checkerboard colouring, is chosen at each vertex of $F$;
  \item \label{mot3}  the arrows are consistent with a local orientation at the vertex. 
  \end{enumerate}

Cycle family graphs arise by relaxing each of the restrictions (\ref{mot1}), (\ref{mot2}) and (\ref{mot3}) that appear in the construction of the Tait graph. For relaxing restriction (\ref{mot1}), we note that there are three natural configurations associated with a $4$-valent vertex    ( \raisebox{-4mm}{\includegraphics[width=1cm]{v1nb}},  \raisebox{-4mm}{\includegraphics[width=1cm]{v2nb}} and  \raisebox{-4mm}{\includegraphics[width=1cm]{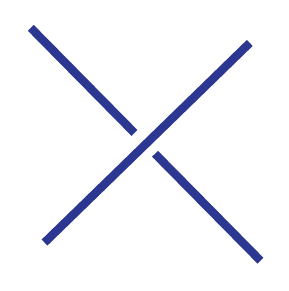}}), rather than two.  We allow all three types of configurations in the construction of a cycle family graph. For relaxing restriction (\ref{mot2}), we allow any of the three types of configuration mentioned above at each vertex. Notice that this means that we no longer require the $4$-regular embedded graphs $F$ to be checkerboard colourable. Finally, we relax restriction (\ref{mot3}) by also allowing the arrows to disagree with a local orientation at the crossing.  Thus cycle family graphs extend the idea of a Tait graph to all embedded graphs.
We now give the formal definition of a cycle family graph.

\bigskip

Let $F$ be a $4$-regular embedded graph thought of as a 2-cell embedding.   A {\em vertex state}  of $v\in \V(F)$ is a choice of one of the following configurations in a neighbourhood of the vertex $v$: 
\[ \includegraphics[width=2cm]{v1nb}\quad \raisebox{10mm}{,} \quad \quad  \includegraphics[width=2cm]{v2nb}\quad \quad\raisebox{10mm}{ or } \quad \quad \includegraphics[width=2cm]{v3nb}\raisebox{10mm}{.}\]
The configurations replace a small neighbourhood of the vertex $v$.
We will refer to the first two of the these vertex states as {\em splits} and the third as a {\em crossing}. Vertex states are sometimes called transitions or transition systems, but here again we choose terminology that is closer to that of knot theory.

An {\em arrow marked vertex state} of $v$ consists of a vertex state equipped with exactly two $v$-labelled arrows. Each arrow is placed on one of the positions indicated below and may point in either direction. 
\[
\includegraphics[height=15mm]{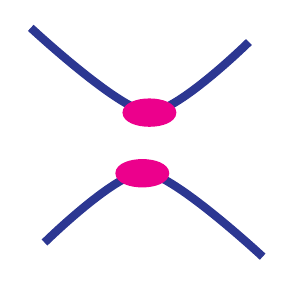}\raisebox{7mm}{,} 
\quad \quad\includegraphics[height=15mm]{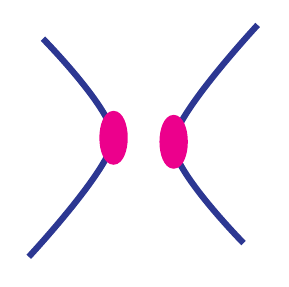}\raisebox{7mm}{,} 
\quad \quad\raisebox{7mm}{or}\quad\quad \includegraphics[height=15mm]{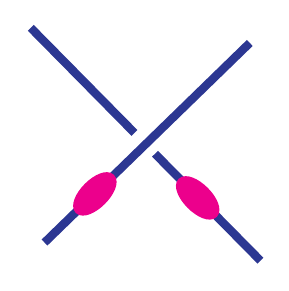}\raisebox{7mm}{.}
\]


A {\em graph state} $s$ of $F$ is a choice of vertex state at each vertex of $F$, and an {\em arrow marked graph state} $\vec{s}$  of $F$ is a choice of arrow marked vertex state at each vertex of $F$.  Note that each graph state corresponds to a specific family of edge-disjoint cycles in $F$, and this family is independent of embedding (although different embeddings of $F$ will generally use different vertex states to generate the same family of disjoint cycles).

\begin{definition}
Let $F$ be a $4$-regular embedded graph, and let $\vec{s}$ be an arrow marked graph state of $F$. Regard 
$\vec{s}$ as an arrow presentation of an embedded graph by viewing each component of $\vec{s}$ as a circle marked by the labelled arrows arising from the arrow marked vertex states.  Denote this new graph by $F_{\vec{s}}$, and, because the vertices of $F_{\vec{s}}$ arise from a family of disjoint cycles of $F$, we call $F_{\vec{s}}$ a {\em cycle family graph} of $F$.  Also note that there is a natural identification between the vertex set of $F$  the edge set of $F_{\vec{s}}$.  We denote the set of all cycle family graphs of a $4$-regular embedded graph $F$ by $\mathcal{C}(F)$.
\end{definition}

\begin{example}\label{e.cfg2}
As an example, on the projective plane, if $F=$ \raisebox{-10mm}{\includegraphics[height=20mm]{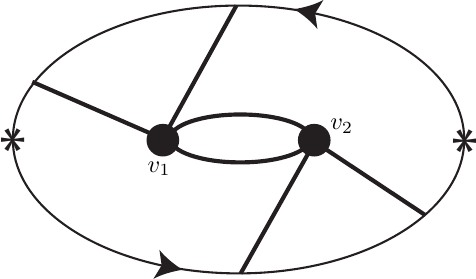}}, then one of the  arrow marked graph states $\vec{s}$ of $F$ gives the following cycle family graph:
 \[ \raisebox{10mm}{$F_{\vec{s}}=$} \;\; \includegraphics[height=20mm]{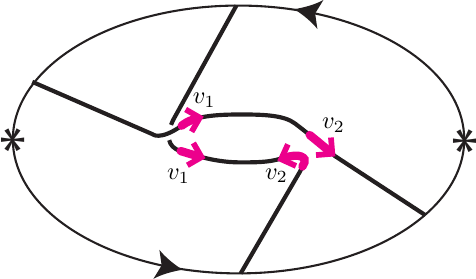}\;\;
  \raisebox{10mm}{$=$} \;\; \includegraphics[height=20mm]{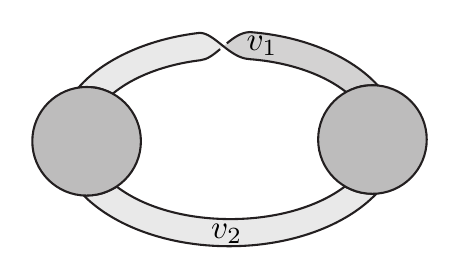}\;\;
  \raisebox{10mm}{$=$} \;\; \includegraphics[height=20mm]{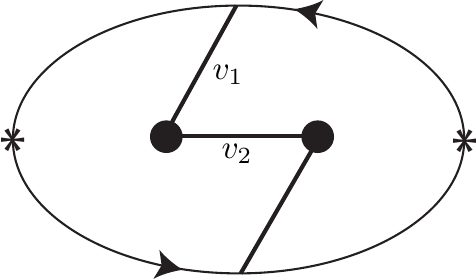}\;\raisebox{10mm}{.}
   \]
\end{example}
Although in Example~\ref{e.cfg2}, $F$ and $F_{\vec{s}}$ are both embedded in the same surface, in general this will not happen, as in Example~\ref{e.cfg}. 
\begin{example}\label{e.cfg}
As a second example of the construction of cycle family graphs, the reader can verify that the complete set of cycle family graphs of 
\;\raisebox{-7mm}{\includegraphics[height=15mm]{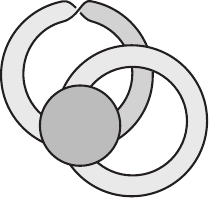}} \;\; contains only
\raisebox{-7mm}{\includegraphics[height=15mm]{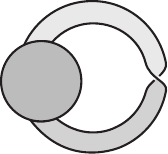}} \;, \;\;\;\raisebox{-7mm}{\includegraphics[height=15mm]{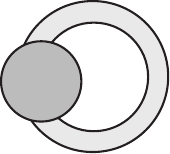}}\; and \;\; \raisebox{-7mm}{\includegraphics[height=15mm]{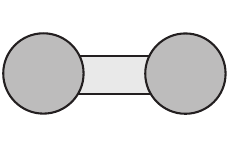}}\;.

\end{example}

We will say that two arrow marked vertex states are {\em equivalent} if 
we can obtain one from the other by reversing the direction of both arrows in the pair.
Figure~\ref{Fig:dualityjoins} illustrates this in the case of splits.  We refer to the arrow marked vertex states on the left of Figure~\ref{Fig:dualityjoins}, with the arrows pointing in opposite directions, as \emph{flat} arrowed marked vertex states, and the arrow marked vertex states on the right of Figure~\ref{Fig:dualityjoins}, with the arrows pointing in the same direction, as \emph{twisted} arrow marked vertex states. Note that these types of states are defined with respect to the embedding of the arrow marked vertex states. 
Furthermore, we say that two arrow marked graph states $\vec{s}$ and $\vs$ of $F$ are {\em equivalent} if, for each vertex of $F$, the arrow marked vertex states of $\vec{s}$ and $\vs$ at that vertex are equivalent when thought of in terms of arrow presentations.
\begin{figure}

\[    \raisebox{-4mm}{\includegraphics[height=10mm]{a1}} \;\; \sim \;\; 
 \raisebox{-4mm}{\includegraphics[height=10mm]{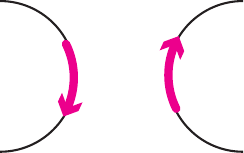}}\quad  \quad \quad \quad\raisebox{0mm}{and}
 \quad  \quad \quad \quad
  \raisebox{-4mm}{\includegraphics[height=10mm]{a2}} \;\; \sim \;\; 
 \raisebox{-4mm}{\includegraphics[height=10mm]{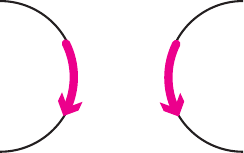}} \;\;  
 \]
\caption{Equivalent arrowed vertex states, with flat arrows on the left, and twisted arrows on the right.}
\label{Fig:dualityjoins}
\end{figure}

\begin{lemma} 
If $F$ is a $4$-regular embedded graph and  $\vec{s}$ and $\vs$ are two equivalent arrow marked graph states of $F$, then $F_{\vec{s}}=F_{\vs}$ as embedded graphs.  
\end{lemma}
\begin{proof}
We need to show that the arrow presentations, $F_{\vec{s}}$ and $F_{\vs}$ are equivalent, that is, $F_{\vs}$
can be obtained from $F_{\vec{s}}$ by reversing the direction of some of the pairs of arrows with the same labels and by homeomorphism of the cycles. To do this, it is enough to show that the arrow marked graph state $\vs$ can be obtained from $\vec{s}$ by reversing the direction of some of the pairs of arrows with the same labels and by homeomorphism of the cycles. The fact that this is indeed the case is easily verified by checking the defining relations of equivalent arrow marked graph and vertex states.
\end{proof}

We note that the converse of the above lemma is false as non-equivalent arrow marked graph states may give rise to equivalent cycle family graphs.

\begin{proposition} 
There are at most $6^{v(F)}$ distinct cycle family graphs of a $4$-regular embedded graph $F$.
\end{proposition}

At certain points in this paper we will  be particularly interested \emph{duality states}. These are  arrow marked graph states which arise by restricting all of the vertex states to splits with flat arrows. In particular, Tait graphs arise from special duality states, as follows.

\begin{proposition}\label{p.cyrad}
Let $F$ be a checkerboard coloured embedded $4$-regular graph.  Then there exist arrow marked graph states $\vec{b}$ and $\vec{w}$ of $F$ such that $F_{\vec{b}}=F_{bl}$ and $F_{\vec{w}}=F_{wh}$. Moreover, both $\vec{b}$ and $\vec{w}$  are duality states.
\end{proposition}
\begin{proof}
 Construct an arrow marked graph state of $F$ by choosing the duality state consisting of flat splits where the split at each vertex results in arcs that follow the boundaries of the black regions at that vertex.
The resulting ribbon graph $F_{\vec{b}}$ is exactly the Tait graph $F_{bl}$ since the cycles follow the boundaries of the black regions (giving one vertex in each black region), and, since the arrow marked state is flat, there is an edge added whenever two black regions meet at a vertex as prescribed by the definition of $F_{bl}$. The whiteface result is proved analogously by choosing the flat splits at each vertex that follow the boundaries of the white regions.
\end{proof}

\subsection{Twisted duals and cycle family graphs}

We are now ready for the first of our main theorems. We will begin by showing that all of the cycle family graphs of a $4$-regular embedded graph are twisted duals, thus generalizing the well known property of Equation~\ref{taitemedial}. We will then prove the converse of this result, that twisted duals are exactly the cycle family graphs of a medial graph. This converse generalizes Proposition~\ref{cyclefamilydual}.

\begin{theorem}\label{t.radmed1}
 If $F$ is a $4$-regular embedded graph and $F_{\vec{s}}$ and $F_{\vs} $ are two cycle family embedded graphs, then $F_{\vec{s}}$ and $F_{\vs} $ are twisted duals.
\end{theorem}
\begin{proof}
It suffices to show that if the arrow marked graph states $\vec{s}$ and $\vs$ differ at exactly one vertex then $F_{\vec{s}}$ and $F_{\vs}$ are twisted duals. To show this, assume that $\vec{s}$ and $\vs$ differ at the vertex $v\in \V(F)$ and that $e_v$ is the label of the edge in $F_{\vec{s}}$ and the corresponding edge in $F_{\vs}$ that arises from the pair of $v$-labelled arrows in $\vec{s}$ and $\vs$. We then need to show that $\g(F_{\vec{s}}, e_v)=(F_{\vs}, e_v)$ for some $\g \in \fG$. 
To show this, we may assume without loss of generality (as $\fG$ is a group) that the arrow marked vertex state at $v$ in $\vec{s}$ is a split with flat arrow markings, so that
locally, (the arrow presentation of) $F_{\vec{s}}$ is \raisebox{-6mm}{\includegraphics[height=12mm]{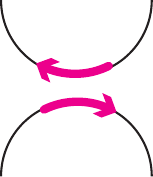}}.  
Then, in $\vs$, the arrow presentation of $F_{\vs}$ is locally one of 
\[     \raisebox{-6mm}{\includegraphics[height=12mm]{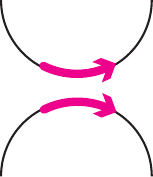}} \quad = \tau\left(  \; \raisebox{-6mm}{\includegraphics[height=12mm]{a9}} \; \right),  \quad \quad 
 \raisebox{-6mm}{\includegraphics[height=12mm]{a1}} \quad = \delta \left(  \; \raisebox{-6mm}{\includegraphics[height=12mm]{a9}} \; \right),  \quad \quad 
  \raisebox{-6mm}{\includegraphics[height=12mm]{a2}} \quad = \tau\delta \left(  \; \raisebox{-6mm}{\includegraphics[height=12mm]{a9}} \; \right), 
\]
\[     \raisebox{-6mm}{\includegraphics[height=12mm]{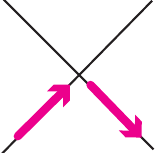}} \quad = \delta\tau\left(  \; \raisebox{-6mm}{\reflectbox{\includegraphics[height=12mm]{a9}}} \; \right)  \quad \quad \text{or} \quad \quad
 \raisebox{-6mm}{\includegraphics[height=12mm]{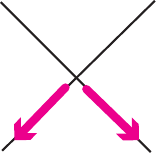}} \quad = \tau\delta\tau\left(  \; \raisebox{-6mm}{\reflectbox{\includegraphics[height=12mm]{a9}}} \; \right)  \;\raisebox{0mm}{,}
\]
where the arrow presentations are all identical outside of the region shown in the diagrams.  
Thus we have shown that when $\vec{s}$ and $\vs$ differ at $v$, then  $\g(F_{\vec{s}}, e_v)=(F_{\vs}, e_v)$ for some $\g \in \fG$, as required.
\end{proof}

\medskip

Suppose that $G$ and $H$ are embedded graphs  with the property that $H=(G_m)_{\vec{s}}$. Then Theorem~\ref{t.radmed1} tells us that $H$ and $G$ are necessarily twisted duals. Below, in  Theorem~\ref{t.radmed}, we will see that the converse of this statement also holds. That is, if $H$ and $G$ are  twisted duals, then $H=(G_m)_{\vec{s}}$, for some state $\vec{s}$.  Theorem~\ref{t.radmed} combines Theorem~\ref{t.radmed1} with this converse to say that $G$ and $H$ are twisted duals if and only if they are both cycle family graphs of the same embedded graph $G_m$.  This can be stated more concisely as $\mathcal{C}(G_m) = Orb(G)$.

 Theorem~\ref{t.radmed} is one of the main results of this paper, since it provides a characterization of twisted duals in terms of medial graphs, and states that the  ribbon group action on an embedded graph $G$ generates precisely the set of cycle family graphs of medial graph of $G$.

\begin{theorem}\label{t.radmed}
If $G$ is an embedded graph, then the set of cycle family graphs of $G_m$ is precisely the set of all twisted duals of $G$, {\em i.e.} 
\[
\mathcal{C}(G_m) = Orb(G).
\]
\end{theorem}

\begin{proof}  By Theorem~\ref{t.radmed1}, if $H \in \mathcal{C}(G_m)$, then $H \in Orb((G_m)_{\vec{s}}))$ for any arrow marked state $\vec{s}$.  In particular this is true for the arrow marked state $\vec{b}$ 
corresponding to the black face graph that is guaranteed to exist by Proposition~\ref{p.cyrad}, so that $(G_m)_{\vec{b}} =(G_m)_{bl}= G$.  Thus $\mathcal{C}(G_m) \subseteq Orb(G)$.

To show that $\mathcal{C}(G_m) \supseteq Orb(G)$, we need to show that if $G$ and $G^{\prime}$ are twisted duals of one another, then they are both cycle family ribbon graphs of the medial ribbon graph $G_m$ of $G$. To show this,
consider locally an edge $e$ of $G$:  \raisebox{-5mm}{\includegraphics[height=10mm]{a7}} .  
At the edge $e$, the twisted dual will assume one of the following six forms.
\[ \includegraphics[height=12mm]{a7}\;\raisebox{6mm}{,} \quad \quad  \includegraphics[height=12mm]{a2}\;\raisebox{6mm}{,} \quad  \quad\includegraphics[height=12mm]{a9}\;\raisebox{6mm}{,} \quad \quad \includegraphics[height=12mm]{a10}\;\raisebox{6mm}{,} \quad \quad \raisebox{12mm}{\includegraphics[height=12mm, angle=-90]{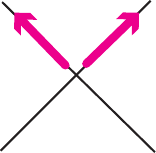}}\;\raisebox{6mm}{,} \quad \quad \raisebox{6mm}{ or }  \quad
\includegraphics[height=12mm]{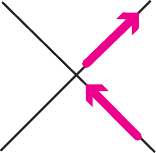} \;\raisebox{6mm}{.} \]
To prove the result it is enough to show that each of these six forms arise as an arrow marked vertex state of $v_e\in \V(G_m)$. (This is since the local configurations at the edges of the arrow presentation of $G$ are connected to each other in the same way that the local configurations at the vertices of $G_m$ are connected to each other to form the cycles.)

Now the vertex $v_e$ of $G_m$ is embedded in the edge $e$ of $G$ thus: \raisebox{-6mm}{\includegraphics[height=12mm]{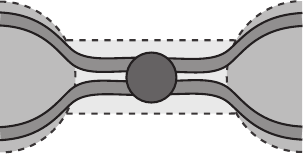}}.
We then see that the possible arrow marked vertex states of $G_m$ at $v_e$ are 
\begin{multline*}
\includegraphics[height=12mm]{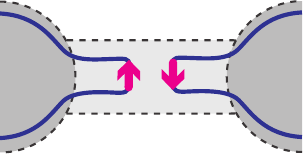}\;\raisebox{6mm}{,} \quad \quad  \includegraphics[height=12mm]{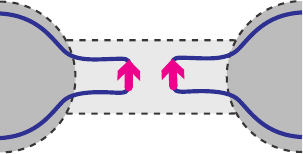}\;\raisebox{6mm}{,} \quad  \quad\includegraphics[height=12mm]{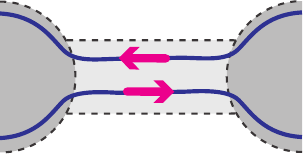}\;\raisebox{6mm}{,} \\  \includegraphics[height=12mm]{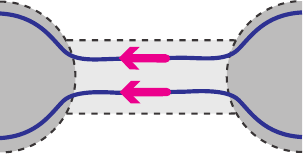}\;\raisebox{6mm}{,}\quad \quad \includegraphics[height=12mm]{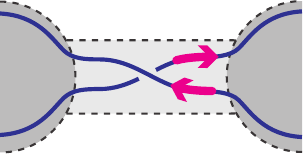}\;\raisebox{6mm}{,} \quad \quad \raisebox{6mm}{and}  \quad
\includegraphics[height=12mm]{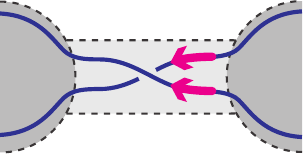}\;\raisebox{6mm}{,} 
\end{multline*}
as required. 
Thus, if $H \in Orb(G)$, then $H$ is a twisted dual of $G$, and by the above $H \in \mathcal{C}(G_m)$ and $Orb(G) \subseteq \mathcal{C}(G_m)$.  Thus $Orb(G) = \mathcal{C}(G_m)$ as desired.
\end{proof}

We note that Theorem~\ref{t.radmed1} does not follow from Theorem~\ref{t.radmed} since, unlike the planar case, there are $4$-regular embedded graphs that do not arise as the medial graph of an embedded graph. (To see this, observe that there are three embedded graphs with exactly one edge, giving rise to three medial graphs with two edges. However, as there are six $4$-regular embedded graphs  with two edges, not every $4$-regular embedded graph can be an embedded medial graph.)

It follows from Proposition~\ref{p.cyrad} that Theorem~\ref{t.radmed} generalizes Equation~\ref{taitemedial} and Propositions~\ref{cyclefamilydual}, since $(G_m)_{bl}=(G_m)_{\vec{b}}$ and $(G_m)_{wh}=(G_m)_{\vec{w}}$ are both in $\mathcal{C}(G_m)$, and $G$ and $G^* = G^{\delta(E(G))}$ are both in $Orb(G)$.

\bigskip

We have seen the relation between cycle family graphs and twisted duals. Since partial duality is a special case of twisted duality that is of independent interest, we specialize our results to partial duality.

\begin{theorem}\label{t.28} ~

\begin{enumerate} 
\item \label{4111} If $F$ is a $4$-regular embedded  graph and $F_{\vec{s}}$ and $F_{\vs}$ are  cycle family embedded graphs, for some $\vec{s}$ and $\vs$ that are duality states, then $F_{\vec{s}}$ and $F_{\vs}$ are partial duals.
\item \label{4112} If $G$ and $G^{\prime}$ are partial duals, and $G_m$ is the embedded medial graph of $G$, then there are duality states $\vec{s}$ and $\vs$ such that $G=(G_m)_{\vec{s}}$ and $G^{\prime}=(G_m)_{\vs}$.
\item \label{4113} $(G_m)_{\vec{s}}$ is a partial dual of $G$ if and only if $\vec{s}$ is a duality state.
\end{enumerate}
\end{theorem} 
\begin{proof}

The first two parts of the theorem can be proved by following the proof of Theorem~\ref{t.radmed} and restricting to partial duality and duality states. Due to this similarity  the proofs are omitted.

The proof of (\ref{4113}) is as follows. Since $G_m$ is a $4$-regular embedded graph and $G=(G_m)_{\vec{s}}$ for some duality state $\vec{s}$, necessity follows by (\ref{4111}). For sufficiency, we may assume without loss of generality that $(G_m)_{\vec{s}}$ is obtained from $G$ by forming the partial dual at a single edge $e$ (so $(G_m)_s=G^{\delta(e)}$). If $e$ is the edge 
\raisebox{-4mm}{\includegraphics[height=8mm]{a1}},
  then the corresponding edge in $(G_m)_{\vec{s}}$ is  \raisebox{-4mm}{\includegraphics[height=8mm]{a3}}, where the graphs are identical except in the region shown. 
It is easily seen that the only way that this configuration can arise from an arrow marked vertex state at $v_e \in \V(G_m)$ is if the arrowed vertex state is a flat split. 
\end{proof}

\bigskip

\subsection{Medial graphs and cycle family graphs}\label{medial and cycle}

In this subsection we will prove another of our main theorems, this one generalizing Equation~\ref{medialtaite} and Proposition~\ref{blacktomedial}. The theorem shows that  cycle 
family graphs extend the essential relations among plane graphs and their medial and Tait graphs. Furthermore, we show that if $F$ is a $4$-regular graph, then its set of cycle family graphs is precisely the set of embedded graphs that have embedded medial graphs isomorphic to $F$ as abstract graphs.  We actually prove a stronger result:  we not only show that the set of cycle family graphs of $F$ give all the embedded graphs whose medial graphs are isomorphic to $F$ as abstract graphs, but we also give specific conditions for when a medial graph is a partial Petrial of $F$.

\begin{theorem}\label{iso to F}
Let $F$ be a $4$-regular embedded graph. Then: 
\begin{enumerate}
\item if  $\vec{s}$ is an arrow marked graph state of $F$, then $(F_{\vec{s}})_m$ and $F$ are isomorphic as abstract graphs;
\item if $\vec{s}$ is a duality state, then $(F_{\vec{s}})_m$ and $F$ are also equivalent as locally embedded maps, or equivalently, are partial Petrials.

\end{enumerate}
\end{theorem}
\begin{proof}
In order to prove the statements it is enough to consider  what happens locally at a vertex $v$ of $F$ in the formation of $(F_{\vec{s}})_m$. We assume that  $v$ is incident to edges labelled $a_v, b_v, c_v,d_v$, and these edges meet $v$  in the cyclic order $(a_v \,b_v\,  c_v \,d_v)$. We will prove the second statement first.

To prove the second statement, assume that $\vec{s}$ is a duality state. Without loss of generality we may assume that the cycles defining the cycle family graph $F_{\vec{s}}$ travel between edges $a_v$ and $b_v$,  and between $c_v$ and $d_v$. This is shown as the first step in the figure below.
\[
\begin{array}{ccccc}
\includegraphics[height=2cm]{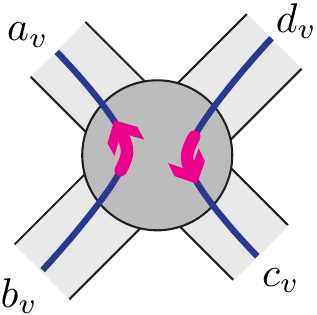} & \raisebox{1cm}{\includegraphics[width=13mm]{arrow}}&
\includegraphics[height=2cm]{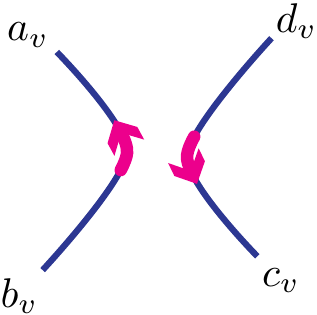} & \raisebox{1cm}{\includegraphics[width=13mm]{arrow}}  & 
\includegraphics[height=2cm]{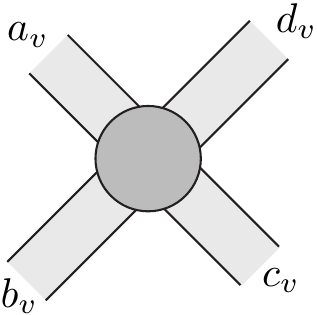} 
\\
F & & F_{\vec{s}} && (F_{\vec{s}})_m
\end{array}
   \]

Now taking the medial graph of $F_{\vec{s}}$ will add  a vertex incident with edges labelled  $a_v, b_v, c_v,d_v$ with the cyclic order $(a_v \,b_v\,  c_v \,d_v)$, as shown in the figure above, possibly with additional twisting of the edges. 
Note that since we do not know if the arcs $(a_v,b_v)$ and $(c_v,d_v)$ are connected to each other, we do not know if the edge $e_v$ of the cycle family graph will embed in the vertex $v$ shown the first figure. This means that when we form the medial graph $(F_s)_m$, as shown in the third step of the figure, we know nothing about the twisting of the edges. The figure is slightly misleading in this respect.   

 Finally, since up to twisting  the edges, $F$ and $(F_{\vec{s}})_m$ are identical at the vertex $v \in F$  and the corresponding vertex in $(F_{\vec{s}})_m$, for each vertex $v$ of $F$, and the endpoints marked  $a_v, b_v, c_v,d_v$ in the figure are connected in the same way in $F$ and in $(F_{\vec{s}})_m$, it follows that these two embedded graphs, when viewed as arrow presentations, have the same sequence of labels on the vertex disc, although possibly not the same directions of the arrows.  Thus, $(F_{\vec{s}})_m$ and $F$ are equivalent as locally embedded maps and  $(F_{\vec{s}})_m = F^{\tau(A)}$ for some $A \subseteq E(G)$. 

This completes the proof of the second statement.

The first statement is proven similarly: all that is needed to adapt the proof is to check that the remaining vertex states  at $v$ lead to an appropriate configuration at the corresponding vertex in $(F_{\vec{s}})_m$. This is seen by the following calculations:
\[
\includegraphics[height=2cm]{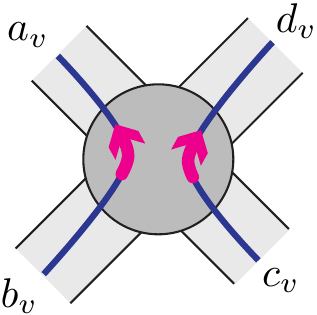} \quad \raisebox{1cm}{\includegraphics[width=13mm]{arrow}}\quad
\includegraphics[height=2cm]{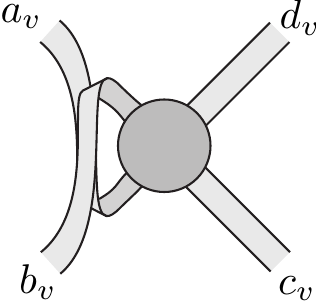}\raisebox{1cm}{,}
\]
\[
\includegraphics[height=2cm]{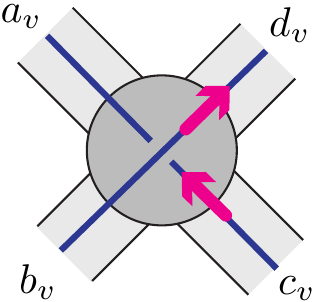} \quad \raisebox{1cm}{\includegraphics[width=13mm]{arrow}}\quad
\includegraphics[height=2cm]{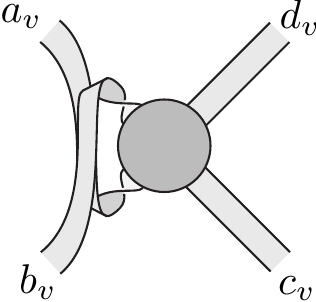}\raisebox{1cm}{,}
\]
\[
\includegraphics[height=2cm]{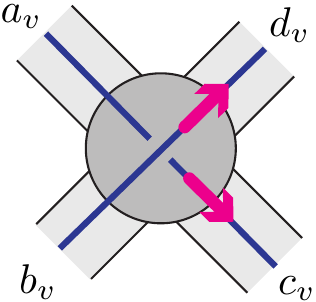} \quad \raisebox{1cm}{\includegraphics[width=13mm]{arrow}}\quad
\includegraphics[height=2cm]{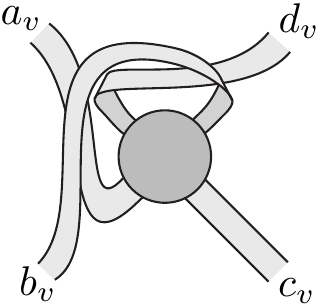}\raisebox{1cm}{.}
\]

There is a natural identification between the vertex set of $F$ and the edge set of $F_{\vec{s}}$, and a natural
identification between the edges set of $F_{\vec{s}}$ and the vertex set of $(F_{\vec{s}})_m$, and hence a natural identification
between the vertices of $F$ and the vertices of $(F_{\vec{s}})_m$.  Thus, the first statement follows by observing that these local configurations imply that a pair of vertices are connected by an edge in $F$ if and only if they are connected by an edge in $(F_{\vec{s}})_m$. 
\end{proof}

Again by Proposition~\ref{p.cyrad}, it is clear that Theorem~\ref{iso to F} generalizes Equation~\ref{medialtaite} and Proposition~\ref{blacktomedial}, since $F_{bl}=F_{\vec{b}}$ and $F_{wh}=F_{\vec{w}}$ for the duality states $\vec{b}$ and $\vec{w}$.

\bigskip

\begin{proposition}\label{p.star}
Let $G$ and $H$ be $4$-regular embedded graphs. If $G$ and $H$ are isomorphic as abstract graphs, then $\mathcal{C}(G) = \mathcal{C}(H)$. In particular, all embeddings of a given $4$-regular graph generate the same set of cycle family graphs.

\end{proposition}
\begin{proof}
Suppose the isomorphism of the underlying abstract graphs of $G$ and $H$ is via $f:E(G)\rightarrow E(H)$. Then an arrow marked state $\vec{s}$ of $G$ induces an arrow marked state $f(\vec{s})$ of $H$ by pairing edges $f(a)$ and $f(b)$ with an arrow from $f(a)$ to $f(b)$ in $f(\vec{s})$  if and only if  $\vec{s}$ pairs edges $a$ and $b$ with an arrow from $a$ to $b$. Thus the arrow marked disks of $G_{\vec{s}}$ are precisely the arrow marked disks of 
$H_{f(\vec{s})}$ under the relabelling given by $f$. Hence $G_{\vec{s}} = H_{f(\vec{s})}\in \mathcal{C}(H)$, and $\mathcal{C}(G) \subseteq \mathcal{C}(H)$. A symmetric argument shows that   $\mathcal{C}(H) \subseteq \mathcal{C}(G)$.
\end{proof}

\bigskip

With the following theorem we are now able to answer the original problem of
finding the exact set of embedded graphs that have medial graphs isomorphic
to a given $4$-regular graph $F$.  

\begin{theorem} \label{cycle to medial}
If $F$ is any $4$-regular graph, then the set of cycle family graphs of any embedding of $F$ is precisely the set of all embedded graphs $G$ such that  $G_m$ and  $F$ are equivalent as abstract graphs, {\em i.e.} if $\widetilde{F}$ is any embedding of $F$, then 
\[
\mathcal{C}(\widetilde{F})= \{G|G_m \cong \widetilde{F}\}= \{G|G_m  \cong F\}.
\]
\end{theorem}

\begin{proof} 
 If $G \in \mathcal{C}(\widetilde{F})$, then $G=\widetilde{F}_{\vec{s}}$, for some state $\vec{s}$, and hence $G_m=(\widetilde{F}_{\vec{s}})_m \cong \widetilde{F} \cong F$, by Theorem~\ref{iso to F}.

Conversely, if $G_m\cong \widetilde{F}\cong F$, then $G_m$ and $F'$ are equivalent as embedded  graphs, for some embedding $F'$ of $F$. By Proposition~\ref{p.cyrad}, $G=(G_m)_{\vec{b}} = F'_{\vec{b}} \in  \mathcal{C}(F')$. Finally, by Proposition~\ref{p.star}, we have $\mathcal{C}(F')=\mathcal{C}(\widetilde{F})$, completing the proof.

\end{proof}

Theorem \ref{cycle to medial} may be interpreted as follows.
\begin{corollary}
If $F$ is any $4$-regular graph, then the set of cycle family graphs of any embedding of $F$ is the set of Tait graphs of all checkerboard colourable embeddings of $F$.

\end{corollary}

\begin{proof}
This follows immediately from Theorem \ref{cycle to medial} since if $F \cong G_m$ for some $G$, then $G_m$ is a checkerboard colourable embedding of $F$, and $G$ is one of its Tait graphs.

\end{proof}

We illustrate Theorem~\ref{cycle to medial} with the following example.
\begin{example} Let $F$ be the abstract graph with one vertex and two loop edges. 
To calculate the set $\{G|G_m  \cong F\}$, observe that if $G_m \cong F$, then $G$ must be a connected embedded graph with one edge. There are three such embedded graphs and it is easy to verify that the medial graph of each of these three embedded  graphs is isomorphic to $F$. Thus 
\[ \{G|G_m  \cong F\} = \left\{   \raisebox{-4mm}{\includegraphics[height=10mm]{cfgex2}}\;, \;\;\raisebox{-4mm}{\includegraphics[height=10mm]{cfgex3}}\;, \;\; \raisebox{-4mm}{\includegraphics[height=10mm]{cfgex4}}  \right\}.  \]
It is worthwhile emphasizing that although there are six distinct embedded graphs isomorphic to $F$ as abstract graphs, only two of these are checkerboard colourable and are therefore medial graphs of some embedded graph. One of these is the medial graph of both the path on one edge and of a single untwisted loop, which are duals of one another; the other is the medial graph of a single twisted loop, which is self-dual.

Now, to construct  $\mathcal{C}(\widetilde{F})$ we need to chose any one of the six embeddings of $F$. For this example we choose an embedding on the Klein bottle and set $\widetilde{F}=  \raisebox{-4mm}{\includegraphics[height=10mm]{cfgex1}} $. (Corollary~\ref{cycleff} below states that  any other choice of embedding of $F$ (including non-checkerboard colourable embeddings)  will result in the same set of cycle family graphs.)   From Example~\ref{e.cfg}, we have 
\[ \mathcal{C}(\widetilde{F})= \left\{   \raisebox{-4mm}{\includegraphics[height=10mm]{cfgex2}}\;, \;\;\raisebox{-4mm}{\includegraphics[height=10mm]{cfgex3}}\;, \;\; \raisebox{-4mm}{\includegraphics[height=10mm]{cfgex4}}  \right\} =\{G|G_m  \cong F\} ,  \]
as is required by the theorem.
\end{example}

Theorem~\ref{cycle to medial} can be used to characterize $4$-regular embedded graphs that are isomorphic as graphs, in terms of cycle family graphs: 
\begin{corollary}\label{cycleff}
Let  $F$ and $F'$ be two $4$-regular embedded graphs. Then  $F$ and $F'$ are isomorphic as abstract graphs ({\em i.e.} $F \cong F'$) if and only if they admit the same set of cycle family graphs ({\em i.e.} $\mathcal{C}(F)= \mathcal{C}(F')$).
\end{corollary}
\begin{proof}
The result follows since, by Theorem~\ref{cycle to medial},
\[  \mathcal{C}(F)= \{G|G_m \cong F\} = \{G|G_m \cong F'\}=\mathcal{C}(F'). \]
\end{proof}

A notable special case of corollary \ref{cycleff} occurs when $F$ and $F'$ are checkerboard colourable.
\begin{corollary} \label{cycleff cor}
Let $F$ and $F'$ be $4$-regular, checkerboard colourable,  embedded graphs, with $F \cong F'$, then the four Tait graphs, $F_{wh}$, $F_{bl}$,  $F'_{wh}$ and $F'_{bl}$, of $F$ and $F'$, are all twisted duals of one another.
\end{corollary}
\begin{proof}
This follows immediately from Corollary~\ref{cycleff} and Theorem~\ref{t.radmed1}. 
\end{proof}

Theorem \ref{cycle to medial} also leads to a second characterization of the orbit of an embedded graph under the ribbon group action. The following result is one our our main results. It states that two embedded graphs are twisted duals if and only if their medial graphs are isomorphic as abstract graphs. This provides a way to lift isomorphism problems involving a medial graph $G_m$ to the embedded graph $G$ itself. 
\begin{theorem}\label{c.orbg}
If $G$ is an embedded graph and $\vec{s}$  is an arrow marked state of the embedded medial graph $G_m$, then $((G_m)_{\vec{s}})_m \cong G_m$, {\em i.e.}
\[
Orb(G) = \{H:H_m \cong G_m\}.
\]
\end{theorem}

\begin{proof}
This follows immediately from Theorems~\ref{t.radmed} and~\ref{cycle to medial}.
\end{proof}

\begin{example}
To illustrate Theorem~\ref{c.orbg}, let $G$ denote the  plane digon. The orbit of $G$ was given in Subsection \ref{TwDu} in  Figure~\ref{Fig:orbit}. It is readily verified that every embedded graph in  $Orb(G)$ has a  medial graph isomorphic to  $G_m$. On the other hand, if $H$ has a medial graph isomorphic to $G_m$ (so $H\in \{H\;|\;H_m \cong G_m\}$), then $H$ must have two edges. By calculating  the medial graphs of each embedded graph with two edges, one can easily check that  any such $H\in Orb(G)$.
\end{example}

\subsection{Partial duals, partial Petrials, and Medial graphs}\label{ss.mgpd}

As mentioned previously, geometric duality can be completely characterized in terms of equivalence of embedded medial graphs:
\[  \{ G, G^*\} = \{H \;|\; H_m =G_m\}. \]
We have seen above that the set of twisted duals of an embedded graph $G$  arises as the set of graphs with the same medial graph  as $G$  when the medial graphs are considered as abstract graphs:
 \[  Orb(G)= \{H \;|\; H_m \cong G_m\}. \]
From these result, we can posit that notions of duality are generated by notions of graph equivalence, that is, we can take the point of view that the set $\{H \;|\; H_m \sim G_m\}$ describes a set of generalized dual graphs for each graph equivalence $\sim$.  In the section we consider the duality generated by considering $H_m$ and $G_m$ to be equivalent as locally embedded maps.  In particular, in Theorem~\ref{c.orbd} we will show that for each embedded graph $G$, 
\[ Orb_{(\delta )}(G)=\{H\;|\;G_m \text{ and } H_m \text{ are equivalent as locally embedded maps}\},\]
and therefore equivalence  as locally embedded maps induces the concept of partial duality.
Remarkably, as a consequence of this, we shall see that partial duality is intimately connected to the partial Petrials of medial graphs. 
\medskip

Let $F$ be a $4$-regular embedded graph. We will denote the subset of cycle family graphs of $F$ that are obtained from duality states by  $\mathcal{C}_{\delta} (F)$, that is
\[ \mathcal{C}_{\delta} (F) := \{  F_{\vec{s}} \; |\; \vec{s} \text{ is a duality state of } F \} .\]

Theorem~\ref{dual to medial} states that the set $\mathcal{C}_{\delta} (F)$ of cycle family graphs generated by duality states characterizes the set of those embedded graphs whose whose medial graphs are equivalent to $F$ as locally embedded maps. 
The theorem should be compared with Theorem~\ref{cycle to medial}. 

\begin{theorem} \label{dual to medial}
If $F$ is any $4$-regular embedded graph, then the subset of cycle family graphs of $F$ that are generated by all duality states is precisely the set of all embedded graphs $G$ such that $G_m$ and  $F$ are equivalent as locally embedded maps, {\em i.e.} 
\[\begin{split}
\mathcal{C}_{\delta}(F)= \{G|G_m \text{ and } F \text{ are equivalent as locally embedded maps}\} \\   \quad \quad=  \{G\;|\;(G_m)^{\tau(A)} = F, \text{ for some } A\in E(G_m)  \}.
\end{split}\]
\end{theorem}

\begin{proof}
The proof is similar to the proof of Theorem~\ref{cycle to medial}.  If $G \in \mathcal{C}_{\delta}(F)$, then $(G_m)^{\tau(A)} = F$,  for some  $A\in E(G_m)$,  by  Theorem~\ref{iso to F}.  On the other hand, if $G$ is an embedded graph such that  $G_m$ and $F$ are equivalent as locally embedded maps, then $G_m$ is checkerboard colourable with the black faces containing the vertices of $G$, where we view $G_m$ as being embedded in the ribbon graph $G$.  We take the arrowed state $\vec{b}$ of $G_m$, which is guaranteed to exist by Proposition~\ref{p.cyrad}, so that $(G_m)_{\vec{b}} = (G_m)_{bl}=G$.  
Since $(G_m)^{\tau(A)} = F$,  for some  $A\in E(G_m)$, by half-twisting some of the edges of $G_m$,  the arrowed state $\vec{b}$ of $G_m$ induces an arrow marked  state $\vec{\beta}$  in $F$. This induced arrow marked state is a duality state since   $\vec{b}$ was. Moreover, $G=F_{\vec{\beta}}$ since $\vec{b}$ and $\vec{\beta}$ are equivalent as arrow presentations. Thus $G \in \mathcal{C}_{\delta}(F)$ as required.
\end{proof}

Theorem \ref{dual to medial} can be used to characterize $4$-regular embedded graphs that are equivalent as locally embedded maps: 

\begin{corollary}\label{dualff}
Let  $F$ and $F'$ be two $4$-regular embedded graphs, then  $F$ and $F'$ are equivalent as locally embedded maps if and only if  $\mathcal{C}_{\delta}(F)= \mathcal{C}_{\delta}(F')$.
\end{corollary}
\begin{proof}
The result follows since, by Theorem~\ref{dual to medial},
\[  \mathcal{C}_{\delta}(F)= \{G|(G_m)^{\tau(A)} = F, \text{ for some } A\in E(G_m)\} = \{G|(G_m)^{\tau(A)} = F', \text{ for some } A\in E(G_m)\}=\mathcal{C}_{\delta}(F'). \]
\end{proof}

In parallel with Corollaries \ref{cycleff} and \ref{cycleff cor}, a special case of Corollary \ref{dualff} occurs when $F$ and $F'$ are checkerboard colourable.
\begin{corollary}
Let $F$ and $F'$ be $4$-regular, checkerboard colourable,  embedded graphs, such that $F$ and $F'$ are equivalent as locally embedded maps. Then the four Tait graphs, $F_{wh}$, $F_{bl}$,  $F'_{wh}$ and $F'_{bl}$,  are all partial duals.
\end{corollary}
\begin{proof}
This follows immediately from Corollary~\ref{dualff} and Theorem~\ref{t.28}. 
\end{proof}

The following result should be compared with Theorem~\ref{c.orbg}. It states that  two embedded graphs are partial duals if and only if their medial graphs are equivalent as locally embedded maps. This provides a characterization of the partial duals of an embedded graph in terms of medial graphs. 

\begin{theorem}\label{c.orbd}
If $G$ is an embedded graph and $\vec{s}$  is a duality state of the embedded medial graph $G_m$, then $((G_m)_{\vec{s}})_m$ and  $G_m$ are equivalent as locally embedded maps, {\em i.e.}
\[
Orb_{( \delta)}(G) = \{H \; |\; G_m \text{ and } H_m \text{ are equivalent as locally embedded maps}\} .
\]
\end{theorem}

\begin{proof}
This follows immediately from Theorems~\ref{t.radmed} and~\ref{cycle to medial}.

\end{proof}

The following corollary to Theorem \ref{c.orbd} now gives a remarkable connection between partial duals and partial Petrials, saying that the partial duals of a graph $G$ are precisely the graphs whose medial graphs are partial Petrials of $G_m$.

\begin{corollary}\label{dual petrials}
Let $G$ be an embedded graph. Then
\[
Orb_{( \delta)}(G) = \{H \; |\; G_m \text{ and } H_m \text{ are partial Petrials}\}, 
\]
\emph{i.e.} two graphs G and H are partial duals if and only if their medial graphs are partial Petrials.

\end{corollary}

\begin{proof}
This follows immediately from Theorem \ref{c.orbd} and that, by Proposition \ref{part comb}, two graphs are equivalent as locally embedded maps if and only if they are partial Petrials of one another.
\end{proof}

Theorem \ref{c.orbd} may also be interpreted in terms of Tait graphs, as follows.

\begin{corollary}  Let $G$ be an embedded graph, and suppose $F$ is isomorphic to $G_m$ as abstract graphs. Then

\[
Orb_{( \delta)}(G) = \{\text{Tait graphs of embeddings }  \tilde{F} \text{ of } F \; |\; \tilde{F} \text{ is a checkerboard colourable partial Petrial of } G_m\}.
\]
\end{corollary}

\begin{proof}
If $H$ is a partial dual of $G$, then by Corollary \ref{dual petrials}, $H_m =G_m^{\tau(A)}$ for some $A \in E(G)$.  But since by Proposition \ref{part comb}, $G_m$ and $G_m^{\tau(A)}$ are equivalent as locally embedded maps, and hence as abstract graphs, $G_m^{\tau(A)}$ is also equivalent to $F$ as an abstract graph.  Thus $G_m^{\tau(A)}$ is a checkerboard colourable embedding of $F$ that is a partial Petrial of $G_m$, and $H$ is a Tait graph of it.  Conversely, if $\tilde{F}$ is a checkerboard colourable embedding of $F$ that is a partial Petrial of $G_m$, and $H$ is one of its Tait graphs, then $H_m = G_m^{\tau(A)}$, and hence, by Corollary \ref{dual petrials}, $H$ is a partial dual of $G$.

\end{proof}

We can also use the relationships among partial duals, twisted duals and medial graphs, together with a result of Las~Vergnas \cite{Las78} and Kotzig \cite{Ko66},
to deduce some properties of the orbits under the ribbon group action.

\begin{proposition}
Let $G$ be a plane graph.  Then 
\[ \mathrm{max}\{  v(H)\; |\; H\in Orb(G)\}  =  \mathrm{max}\{  v(H)\; |\; H\in Orb_{( \delta)}(G)\}  .\]
\end{proposition}
\begin{proof}
Las Vergnas' Proposition~6.1 from \cite{Las78} implies that the maximum number of circuits  in any duality state of $G_m$ is equal to  the maximum number of circuits  in any  state of $G_m$. Since the cycles in the states of $G_m$ form the vertices of the cycle family graphs we have
\[ \mathrm{max}\{ v(H) \;|\; H\in \mathcal{C}_{\langle \delta\rangle}(G_m) \}  = \mathrm{max}\{ v(H) \;|\; H\in \mathcal{C}(G_m) \}. \]
The result then follows since $ \mathcal{C}_{\delta}(G_m) = Orb_{( \delta)}(G)$, by Theorem~\ref{dual to medial} and Theorem~\ref{c.orbd}; and since $ \mathcal{C}(G_m) = Orb(G)$, by Theorem~\ref{cycle to medial} and Theorem~\ref{c.orbg}.
\end{proof}

The following corollary relates the number of spanning trees of a ribbon graph $G$ and of its dual $G^*$ to the number of bouquets ({\em i.e.} embedded graphs with exactly one vertex) in $Orb_{( \delta)}(G)$.

\begin{proposition}
Let $G$ be a graph embedded in the plane, the torus or the real projective plane. In addition, let $\mathcal{B}_{(\delta)}(G) $ denote the  number of bouquets in $Orb_{( \delta)}(G)$.  Then $\mathcal{B}_{(\delta)}(G) $ is bounded above by the total number of spanning trees in $G$ and  $G^*$. Moreover, 
\[   \mathcal{B}_{(\delta)}(G)   \leq 2T(G;1,1),  \]
where $T(G;x,y)$ is the Tutte polynomial of $G$.
\end{proposition}
\begin{proof}
By Las Vergnas' Corollary~2.4 from \cite{Las78}, every duality state of $G_m$ that contains exactly one cycle corresponds to a unique spanning tree in $G$ or $G^*$. (The plane case of this result is due to Kotzig \cite{Ko66}.) 
Moreover,  every duality state of $G_m$ that contains exactly one cycle gives rise to a (not necessarily distinct) cycle family graph in  $\mathcal{C}_{\delta}(G_m)$ that has exactly one vertex.
Therefore,
\[ | \{ \text{spanning trees of } G \text{ or } G^* \} | \geq | \{   H\in \mathcal{C}_{\delta}(G_m) \; | \; H \text{ has exactly one vertex}  \}  |.  \]
Since $\mathcal{C}_{\delta}(G_m) =  Orb_{( \delta)}(G)$,
by Theorem~\ref{dual to medial} and Theorem~\ref{c.orbd}, it follows that 
\[ | \{ \text{spanning trees of } G \text{ or } G^* \} | \geq | \{   H\in Orb_{( \delta)}(G) \; | \; H \text{ has exactly one vertex}  \}  |   =  \mathcal{B}_{(\delta)}(G) .  \]
The result then follows by noting that    the number of spanning trees in a connected graph $G$ is $T(G;1,1)$, and that  $T(G,x,y)=T(G^*;x,y)$, and so 
\[ 2T(G;1,1) =   | \{ \text{spanning trees of } G \text{ or } G^* \} | \geq \mathcal{B}_{(\delta)}(G). \]
\end{proof}

The above result can also be extended to quasi-trees. From Dasbach et al. \cite{Da2}, 
a {\em quasi-tree} is an embedded graph with exactly one boundary component (or face). 
\begin{corollary}
Let $G$ be an embedded graph, then the number of bouquets in $Orb_{( \delta)}(G)$ is bounded above by the number of spanning quasi-trees of $G$.
\end{corollary}
\begin{proof}
From \cite{Mo4}, if $A\subseteq E(G)$, and $A^c=E(G)\backslash A$, then the number of vertices of $G^{\delta(A)}$ is equal to the number of boundary components of $G\backslash A^c$. It then follows that $G^{\delta(A)}$ is a bouquet if and only if $G\backslash A^c$ is a spanning quasi-tree. The result then follows, noting that the partial duals need not be distinct.
\end{proof}

Similar results hold for the orbit of an embedded graph under the action of the subgroup of the ribbon group generated by half-twists. 
\begin{proposition}
Let $G$ be an embedded graph. Then
\begin{enumerate}
\item   $|Orb_{( \tau )}(G)|$ is bounded above by the number of cycles in $G$.
\item If $G$ is bipartite, then $G^{\tau (E(G))}=G$.  Furthermore, (Wilson \cite{Wil79}) if $G$ is orientable regular map, then $G^{\tau (E(G))}=G$ if and only is $G$ is bipartite.
\end{enumerate}
\end{proposition}

\begin{proof}
We will work in the language of ribbon graphs. Both results follow from the observations that $\tau (e)$  changes the ribbon graph by adding a half-twist to the  edge $e$, so up to the equivalence of ribbon graphs, $\tau (A)$ can only act by adding or removing half-twists to cycles, for each $A\subseteq E(G)$. The first result then follows by observing that $\tau (A)$ can only act by changing the orientability of a set of cycles of $G$.

That all embeddings of bipartite graphs are self-Petrial follows because every cycle is even, so its orientability is unchanged if every edge is given a half-twist, while the cyclic order of half-edges incident to each vertex is unchanged. The `if and only if' statement is due to Wilson \cite{Wil79}, and does not in general hold for unorientable embedded graphs. For example, the hemi-dodecahedron in the projective plane is self-Petrial, but not bipartite.

\end{proof}

\subsection{Cycle family graphs and checkerboard colourings}\label{ss.cfgcb}

Here we define an action of the ribbon group ${\fG}^n$ on an arrow marked vertex state of an embedded medial graph $G_m$.  Obviously, $G_m$ could equivalently be any $4$-regular checkerboard colourable embedded graph, but we will typically apply these results to medial graphs. In order to do this we need to be able to distinguish among the three vertex states. If we canonically checkerboard colour $G_m$, then we can distinguish among the vertex states at $v$ as in the following figure.  Here, following knot theory conventions, the graphs are identical outside these local neighbourhoods. 

\begin{center}\begin{tabular}{ccccccc}
\includegraphics[height=15mm]{m1v} & \quad \raisebox{6mm}{$\longrightarrow$} \quad \hspace{1cm} & \includegraphics[height=15mm]{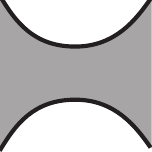} &\hspace{1cm}  & \includegraphics[height=15mm]{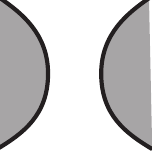} &\hspace{1cm}  &\includegraphics[height=15mm]{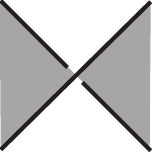} \\
in $G_m$ && white split && black split && crossing.
\end{tabular}\end{center}

We denote the graphs that result from each of these vertex states as follows:

\begin{enumerate}\renewcommand{\labelenumi}{}
  \item $(G_m)_{wh(v)}$ is the embedded graph that results from taking the white split state at $v$,
  \item $(G_m)_{bl(v)}$ is the embedded graph that results from taking the black split state at $v$, and
  \item $(G_m)_{cr(v)}$ is the embedded graph that results from taking the crossing state at $v$.
\end{enumerate} 

This kind of operation at a vertex appears in the literature under several names, including `smoothing' and `transition'.  We use the terminology `split' here to emphasize that the operation `splits', or separates, the regions about a vertex.

We now define an action of the group ${\fG}^n$ on an arrow marked  vertex state.
Let $(G,\ell)\in \calG_{or(n)}$. Then the order $\ell$ of the edges of $G$ induces an order $\hat{\ell} = (v_1, \ldots, v_n)$ of the vertex set of the embedded medial graph $G_m$. We let $(G_m,\hat{\ell})$ denote this canonically checkerboard coloured medial graph equipped with the vertex order induced from $(G,\ell)$.
To define an action on the set of arrow marked states of $G_m$, we let $\vec{s}$ be an arrow marked graph state of $G_m$ and $(\vec{s}, i)$ be the vertex state at the $i^{th}$ vertex $v_i$ in the order $\hat{\ell}$.   
 Then  define $\tau(\vec{s},i)$ to be the pair $(\vs, i)$, where $\vs$ is obtained from $\vec{s}$  by reversing exactly one of the arrows of the arrow marked vertex state at $v_i$ in $\vec{s}$.

Let  $\delta (\vec{s},i)$ be the pair $(\vs,i)$ where $\vs$ is constructed by changing the arrow marked vertex state at $v_i$ in $\vec{s}$ as specified by the following figure:
    \[ \includegraphics[height=15mm]{n1} \includegraphics[height=15mm]{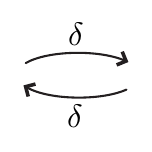} \includegraphics[height=15mm]{n2}\;,\hspace{2cm}
    \includegraphics[height=15mm]{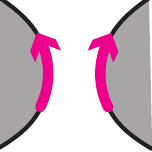} \includegraphics[height=15mm]{dd} \includegraphics[height=15mm]{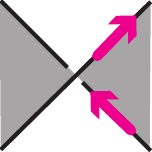}\]
\[\includegraphics[height=15mm]{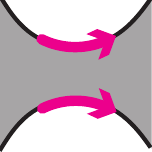} \includegraphics[height=15mm]{dd} \includegraphics[height=15mm]{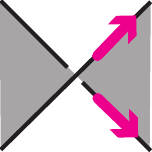}\;\;\raisebox{7mm}{.}\]

Next we want to extend this action to the set of states of $G_m$. 
Let $\mathcal{C}_{or}(G_m,\hat{\ell})$ denote the set of arrow marked graph states of $(G_m,\hat{\ell})$ with the vertex ordering induced by $\ell$.  
(We note that $\mathcal{C}_{or}(G_m,\hat{\ell})$  is equivalent to the set of cycle family graphs of $G_m$ with an edge order induced by $\ell$.) 
Define an action of $\fG^n$ on $\mathcal{C}_{or}(G_m,\hat{\ell})$ by
\[\vg(\vec{s}, \hat{\ell}):= ( \g_1, \g_2,\g_3, \ldots , \g_n )(\vec{s}, \hat{\ell}) = \g_n(  \cdots    \g_3(\g_2( \g_1(\vec{s},v_1), v_2 ),v_3)  \cdots ),v_n)  ,\]
where  $( \g_1, \g_2,\g_3, \ldots , \g_n )\in  \fG^{n}$ and $\vec{s} \in \mathcal{C}_{or}(G_m,\hat{\ell})$.

\begin{proposition}\label{t.medaction}
The action of $\fG^n$ on $\mathcal{C}_{or}(G_m,\ell)$ described above is a group action.
Moreover,  
\[ (G_m, \hat{\ell})_{\vg (\vec{s}, \hat\ell)} = \vg((G_m, \hat{\ell})_{\vec{s}}),\]
where $\vg\in \fG^n$.
\end{proposition}
\begin{proof}
To prove the first part of the proposition we need to check that $ \delta^2(\vec{s},v)= \tau^2(\vec{s},v)= (\tau\delta)^3(\vec{s},v)=1$ and this is easily verified. By checking examples, it is readily verified that no additional relations hold.  The second part of the proposition is tautological when the arrowed states are viewed as arrow presentations of embedded graphs.

\end{proof}

Just as the action of $\fG^n$ on embedded graphs may be interpreted geometrically,  Proposition~\ref{t.medaction} may be given from a geometric perspective as in the following proposition. The proposition below states that the group action on arrow marked medial graphs defined in this subsection is compatible with the ribbon group action.
\begin{proposition}\label{t.medaction geo}
Let $G$ be an embedded graph with embedded medial graph $G_m$, and index the vertices of $G_m$ by $E(G)$.   Let $\Gamma = \prod^6_{i=1}{\g_i(A_i)}$ where the $A_i$'s partition $E(G)$, and the $\g_i$'s are the six elements of $\fG$. If $\vec{s}$ is an arrow marked state of $G_m$, then let $\vG$ be the result of applying $\g_i (\vec{s},v_e)$ if $e \in A_i$.  Then
\[
(G_m)_{\vG} = ((G_m)_{\vec{s}})^{\Gamma}.
\]
\end{proposition}

\subsection{Orbits under the action induced by the subgroups of $\fG$}\label{ss.orbits}

We have given, in Corollary~\ref{cycle to medial},  a characterization of the orbit of an embedded graph $G$ under the entire ribbon group action in terms of medial graphs.  We have also given, in  Theorem~\ref{c.orbd}, a characterization of the orbit of $G$ under the action of the subgroup generated by all elements of the form $(1, \ldots, 1, \delta, 1, \ldots ,1)$,  again in terms of medial graphs.  Here we will provide  geometric characterizations of the orbits under the action of some other important subgroups of the  ribbon group. These characterizations are implicit in Proposition~\ref{t.medaction geo}. A full study of actions of  the many subgroups of a ribbon group on $n$ edges for some class of graphs opens a new and interesting area  investigation.

Recall from Subsection~\ref{ss.notable} that $Orb_{( \delta )}(G)$ is the set of partial duals of $G$, and
 $Orb_{( \tau )}(G)$ is the set of partial Petrie duals of $G$, with similar definitions for the other subgroups of $\fG$.

In this section, we characterize geometrically the orbits $Orb_{( \g )}(G)$ generated by each subgroup of $\fG$ in terms of medial graphs.  In order to do this we need to introduce some further notions of the equivalence of  checkerboard coloured medial graphs. 

Let $F$ be a checkerboard coloured $4$-regular graph. Type $C1$, $C2$, $C3a$, $C3b$ and $C4$ moves are defined in Figure~\ref{fig.cmoves}. In this figure, a $Ci$ move replaces a checkerboard coloured $4$-valent vertex of $F$ with  the checkerboard coloured $4$-valent vertex shown  labelled $Ci$  in the table, where $Ci$ is $C1$, $C2$, $C3a$, $C3b$ or $C4$. (Note that, due to the non-trivial embeddings of the graphs, in the figure the checkerboard colouring is indicated using bold (for the black face) and lighter (for the white face) lines for the boundary components.)
\begin{figure}
\begin{center}
\begin{tabular}{ccccc}
\includegraphics[height=2cm]{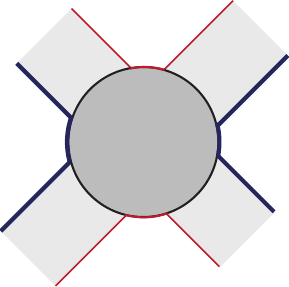} &\hspace{1cm} &\includegraphics[height=2cm]{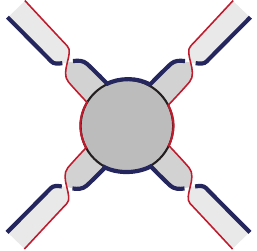}&\hspace{1cm} &\includegraphics[height=2cm]{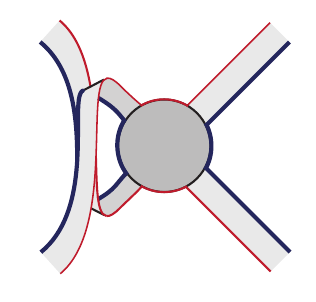} \\
A vertex of  $F$  && Type C1 move && Type C2 move \\ &&\\
\includegraphics[height=2cm]{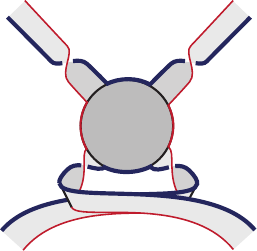} 
& &\includegraphics[height=2cm]{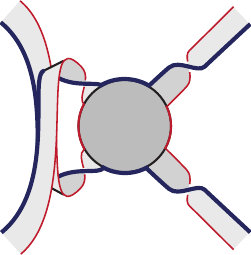} && \includegraphics[height=2cm]{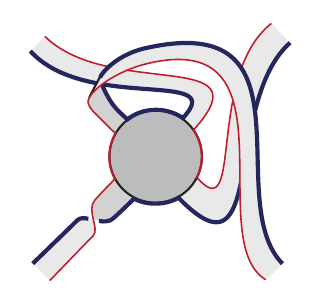} \\
Type C3a move  & & Type C3b move   && Type C4 move 
\end{tabular}
\end{center}
\caption{Moves on a checkerboard coloured embedded graph.}
\label{fig.cmoves}
\end{figure}
We say that two checkerboard coloured $4$-regular graphs $F$ and $H$ are $C1$-equivalent, written $F\sim_{C1} H$, if $F$ and $H$ are related by a finite sequence of $C1$ moves; and similarly for the $C2$ and $C4$ moves. 
We say that $F$ and $H$ are $C3$-equivalent, written $F\sim_{C3} H$, if $F$ and $H$ are related by a finite sequence of $C3a$ and $C3b$ moves.
     It is easy to verify that $Ci$-equivalence, for  $i=1,2,3,4$, is an equivalence relation.

\begin{theorem}\label{orbits}
Let $G$ be an embedded graph. Then
\begin{enumerate}
\item  $Orb(G) = \{H:H_m \cong G_m\}$;
\item $Orb_{( \delta )}(G) 
= \{H \; |\; G_m\sim_{C1} H_m \text{ w.r.t canonical checkerboard colouring}\} $;

$ \quad \quad  \quad \quad = \{H \; |\; G_m \text{ and } H_m \text{ are equivalent as locally embedded maps}\}$; 

\item $Orb_{( \tau )}(G) = \{H \; |\; G_m\sim_{C2} H_m \text{ w.r.t canonical checkerboard colouring }\} $;
\item $Orb_{( \delta\tau )}(G) = \{H \; |\; G_m\sim_{C3} H_m \text{ w.r.t canonical checkerboard colouring}\} $;
\item $Orb_{( \delta\tau\delta )}(G) = \{H \; |\; G_m\sim_{C4} H_m \text{ w.r.t canonical checkerboard colouring}\} $.

\end{enumerate}

\end{theorem}
\begin{proof}
First note than  the $C1$ move is the geometric realization of $(G_m)_{(\vec{b})^{\g (e)}}$ when $\g=\delta$, the $C2$ move is the geometric realization of $(G_m)_{(\vec{b})^{\g (e)}}$ when $\g=\tau$, the $C3a$ when $\g=\tau\delta$, the $C3b$  when $\g=\delta\tau$, and  the $C4$  when $\g=\delta\tau\delta$.
The first result is then  Theorem~\ref{c.orbg} and the second is Theorem~\ref{c.orbd}. 

The remaining results follow easily from Proposition~\ref{t.medaction geo}. To see this we begin by observing that, by  Proposition~\ref{p.cyrad}, $G=(G_m)_{\vec{b}}$. Then if $\g\in \fG$, and $e\in E(G)$ we have, by Proposition~\ref{t.medaction geo}, 
\[ G^{\g (e)} = \left((G_m)_{\vec{b}}\right)^{\g (e)}   = (G_m)_{(\vec{b})^{\g (e)}}.  \]

\end{proof}

The following theorem summarizes how the notions of equivalence of embedded graphs and the notions of duality studied here generate each other.
\begin{theorem}
Let $G$ and $H$ be  embedded graphs. Then there is the following hierarchy of of graph equivalences and notions of duality: 
\begin{enumerate}
\item $G_m$ and $H_m$ are equivalent as embedded graphs if and only if $G$ and $H$ are geometric duals;
\item  $G_m$ and $H_m$ are equivalent as locally embedded maps (or are partial Petrials) if and only if $G$ and $H$ are partial duals;
\item $G_m$ and $H_m$ are equivalent as abstract graphs if and only if $G$ and $H$ are twisted duals.
\end{enumerate}
Furthermore, 
\begin{enumerate}\setcounter{enumi}{3}
\item $G_m$ and $H_m$ are equivalent under $C_3$ moves if and only if $G$ and $H$ are partial trialities;
\item  $G_m$ and $H_m$ are equivalent under $C_4$ moves if and only if $G$ and $H$ are partial Wilsonials.
\end{enumerate}

\end{theorem}

\bigskip



\subsection{Deletion, contraction, and medial graphs}\label{del con medial}

We now discuss how the operations of deletion and contraction interact with the formation of  medial graphs.  Deletion and contraction of non-loop edges of an embedded graph are defined much as for abstract graphs. Let $G$ be an embedded graph. In the language of ribbon graphs it is clear that deleting any edge of $G$ or contracting a non-loop edge will result in a ribbon graph.  When working in the language of cellularly embedded graphs, one has to ensure that deleting or contracting an edge results in a cellularly embedded graph.  In particular, deleting a bridge changes not only the number of components of a graph, but the number of components of the surface it is embedded in.
 Thus deletion and contraction is often best  done by converting to the language of ribbon graphs, carrying out the operation, then converting back to the language of cellularly embedded graphs. Deletion and contraction for arrow presentations can be defined similarly.  Note that $G/e$ and $G^{\delta(e)}-e$ are equivalent when $e$ is not a loop.

Contracting loops requires some care.  We follow Bollob\'{a}s  
and Riordan's definition from Section~7 of \cite{BR2}.
Let $G$ be an embedded graph regarded as a ribbon graph and suppose that  $e$ is a loop of $G$, with $v$ being the vertex of $G$ incident to $e$. Then the ribbon  subgraph $(\{v\}, \{e\})$ of $G$ has either one or two boundary  
components (depending on whether $e$ is a twisted loop or not). Then the ribbon graph $G/e$ is formed from $G$ by attaching a vertex-disc to each of the boundary components of $v\cup e$ in $G$, then  
deleting $e$ and $v$.
 From this definition, it is not hard to see that when $e$ is a loop  
it is also true that the ribbon graphs $G/e$ and $G^{\delta(e)}-e$ are equivalent. Chmutov observed this  in \cite{Ch1}  
and defined the {\em contraction} of any edge $e$ of an  
embedded graph $G$ as
\[   G/e := G^{\delta(e)}-e ,\]
and we follow this convention.
We emphasize the fact that if $e$ is a non-twisted loop whose ends are adjacent to each other on the vertex on which they lie, then contraction creates a vertex. For example, if $G$ is the orientable embedded graph consisting of one vertex and  one edge $e$, then $G/e$ consists of two vertices and no edges.  We also recall that the medial graph of an isolated vertex is again an isolated vertex.

The following proposition  may be readily observed by viewing $G$ and $G_m$ as ribbon graphs as in Figure~\ref{fig.meddef},  and noting how the medial graph is transformed under the indicated operations. 

\begin{proposition}\label{medial deletions}

Let $G$ be an embedded graph with embedded, canonically checkerboard coloured medial graph $G_m$, and let  $e$ be any edge of $G$. Then
\begin{enumerate}
\item  $(G_m)_{bl(v_e)} = (G-e)_m$;
\item $(G_m)_{wh(v_e)}=(G/e)_m$;
\item $(G_m)_{cr(v_e)}$  and  $(G^{\tau(e)}/e)_m$ are twists of each other.

\end{enumerate}

\end{proposition}

\section{ The transition polynomial }\label{sec:transpoly}
The generalized transition polynomial, $q(G; W,t)$,  of \cite{E-MS02} is a multivariate graph polynomial that generalizes Jaeger's transition polynomial~\cite{Jae90}. The transition polynomial assimilates the Penrose polynomial and Kauffman bracket, and agrees with the Tutte polynomial via a medial graph construction.  We will now adapt $q(G; W,t)$ to embedded graphs and determine its interaction with the ribbon group action. 
In \cite{E-MMc}, we use the interaction between twisted duality and the transition polynomial   
to generalize the Penrose polynomial for plane graphs to embedded graphs and determine new properties for it. Also, in  \cite{E-MMd}, by leveraging the relation determined in~\cite{ES} between the generalized transition polynomial and the topological Tutte polynomial of Bollob\'as and Riordan (\cite{BR1, BR2}), we determine new properties of the latter.

\subsection{The topological transition polynomial}
The generalized transition polynomial, $q(G; W,t)$  extends  the  transition polynomial of Jaeger~\cite{Ja90} to arbitrary Eulerian graphs and incorporates pair and vertex state weights. For the current application, however, we will restrict $q$ to $4$-regular embedded graphs (typically medial graphs) and we will only work in the generality needed for our current application. For example, since we will not use pair weights here, we give the weight systems simply in terms of vertex state weights. If an applications arises in the future where the pair weights are needed, because of the restriction to $4$-regular graphs, they may be taken to be square roots of the vertex state weights.  We refer the reader to \cite{E-MS02} or \cite{ES} for further details.

A \emph{weight system}, $W( F )$, of
    any $4$-regular graph $F$ (embedded or not) is an assignment of a weight in a unitary ring $\mathcal{R}$ to every vertex state of $F$.
     (We simply write $W$ for $W(F)$ when the graph is clear from context.)  If $s$ is a state of $F$, then the \emph{state weight} of $s$ is $\omega(s) := \prod_{v \in V(F)} {\omega(v,s)}$, where $\omega(v,s)$ is the vertex state weight of the vertex state at $v$ in the graphs state $s$. Note that a state $s$ consists of a set of disjoint closed curves, and we refer to these as the components of the state, denoting the number of them by $c(s)$.

\begin{definition}\label{transpolydef} Let $F$ be a $4$-regular graph having weight system $W$ with values in a unitary ring $\mathcal{R}$.  Then the state model formulation of the {\em generalized transition polynomial} is 
\[
q(F; W,t)= \sum_{s} {\omega( s )t^{c(s)} },
\]  
where the sum  is over all graph states $s$ of $F$.  
\end{definition}  

We now restrict our attention further to embedded medial graphs and particular weight systems determined by the embeddings. Because of these restrictions, we will call the generalized transition polynomial specialized for this application the \emph{topological transition polynomial}, and define it as follows.

\begin{definition}
 Let $G$ be an embedded graph with embedded medial graph $G_m$.  Define the {\em medial weight system}, $W_m(G_m)$, using the canonical checkerboard colouring of $G_m$ as follows.  A vertex $v$ has state weights given by an ordered triple $(\alpha_v$, $\beta_v$, $\gamma_v)$, indicating the weights of the white split, black split, and crossing state, in that order. We write $(\boldsymbol\alpha, \boldsymbol\beta, \boldsymbol\gamma)$ for the set of these ordered triples, indexed equivalently either by the vertices of $G_m$ or by the edges of $G$.  Then the \emph{topological transition polynomial} of $G$ is:
\[
Q(G, (\boldsymbol\alpha, \boldsymbol\beta, \boldsymbol\gamma), t) :=q(G_m; W_m,t).
\] 
\end{definition}

\begin{proposition}\label{linear recursion Q}
The topological transition polynomial may be computed by repeatedly applying  the following linear recursion relation at each $v\in V(G_m)$, and, when there are no more vertices of degree 4 to apply it to, evaluating each of the resulting closed curves to an independent variable $t$: 
\[
q(G_m, W_m,t)= \alpha_v q((G_m)_{wh(v)}, W_m,t)+ \beta_v q((G_m)_{bl(v)}, W_m,t)+ \gamma_v q((G_m)_{cr(v)}, W_m,t).
\]

Pictorially, this is:
\[ \includegraphics[height=14mm]{m1v}\;\; \raisebox{7mm}{$=$} \;\; \raisebox{7mm}{$\alpha_v$} \; 
\includegraphics[height=14mm]{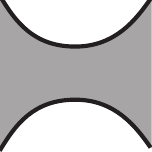}
\raisebox{7mm}{$\;+ \;\; \beta_v$} \;\; 
\includegraphics[height=14mm]{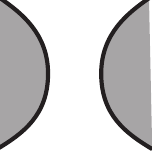}
\raisebox{7mm}{$\;+ \;\;\; \gamma_v$}\;\;\includegraphics[height=14mm]{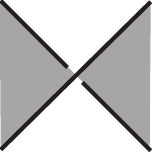}\raisebox{7mm}{.}
 \]

\end{proposition}

\begin{example}
For example, if $G= \raisebox{-4mm}{\includegraphics[height=10mm]{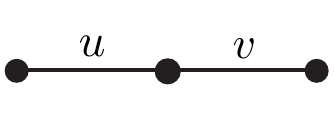}}$, 
then
$G_m=\raisebox{-4mm}{\includegraphics[height=10mm]{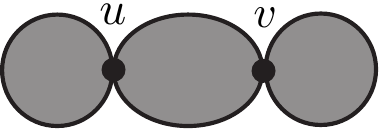}}$ and so
\[\begin{array}{rl}
Q(G; (\bal, \bbe,\bga) , t) =& \alpha_u \; \raisebox{-3mm}{\includegraphics[height=8mm]{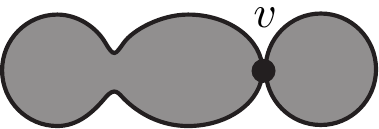}} + \beta_u \; \raisebox{-3mm}{\includegraphics[height=8mm]{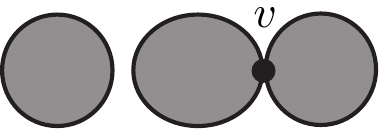}}+ \gamma_u \; \raisebox{-3mm}{\includegraphics[height=8mm]{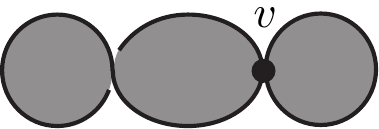}} \\ 
=& \alpha_u \; \left( \alpha_v \;\raisebox{-3mm}{\includegraphics[height=8mm]{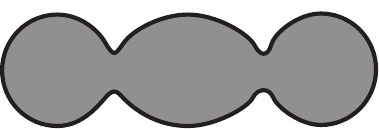}} + \beta_v \; \raisebox{-3mm}{\includegraphics[height=8mm]{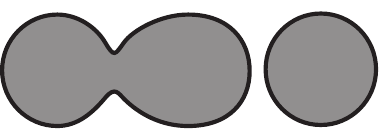}}+ \gamma_v \; \raisebox{-3mm}{\includegraphics[height=8mm]{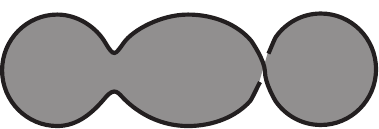}} \right) + \cdots \\
= & \alpha_u \alpha_v t + \alpha_u \beta_v t^2 + \alpha_u \gamma_vt + \cdots
\end{array}\]

\end{example}

We emphasize that in the following theorem the contraction of a loop $e$ (or any other edge $e$ for that matter) is defined by $G/e:=G^{\delta (e)}-e$.

\begin{theorem}\label{t.trans delete contract} Let $G$ be an embedded graph and $e \in E(G)$. Then 
\[
Q(G; (\boldsymbol\alpha, \boldsymbol\beta, \boldsymbol\gamma), t)=\alpha_e Q(G/e; (\boldsymbol\alpha, \boldsymbol\beta, \boldsymbol\gamma), t)+\beta_e Q(G-e; (\boldsymbol\alpha, \boldsymbol\beta, \boldsymbol\gamma), t)+\gamma_e Q(G^{\tau(e)}/e; (\boldsymbol\alpha, \boldsymbol\beta, \boldsymbol\gamma), t).
\]
\end{theorem}

\begin{proof}
 The identity follows  from Propositions~\ref{medial deletions} and~\ref{linear recursion Q} upon observing that twisting an edge of a medial graph does not change the number of cycles in a transition state.
 \end{proof}

\subsection{Twisted duals and the transition polynomial}
We will now see how the topological transition polynomial interacts with the ribbon group action.
The group $\fG= \langle  \delta, \tau    \; |\;   \delta^2, \tau^2, (\tau\delta)^3   \rangle$ is isomorphic to $\fS_3$ via
\[ 
\eta: \tau \mapsto (1\; 3) \quad \text{ and } \quad \eta: \delta \mapsto (1\; 2).  
\]
Furthermore, the symmetric group $\fS_3$  acts on the ordered triple of the weight system at a vertex by permutation. 
This action by $S_3$ on a vertex state weight can be extended to an action of $S_3^n$ on the vertex weight states of medial graphs with $n$ linearly ordered vertices. This can be done by mimicking the approaches used in Subsections~\ref{ss.rga} and~\ref{ss.cfgcb}. We will not formally define this action here and instead  define the order independent analogue of the action, which is more convenient for our applications.
This action allows us to use the ribbon group to modify the medial weight system of an embedded medial graph.  

\begin{definition}\label{permute weights}
Let $G_m$ be a canonically checkerboard coloured embedded medial graph of an embedded graph $G$ with medial weight system $W_m$ (or equivalently $(\bal, \bbe,\bga)$), and vertices indexed by the edges of $G$.  Let $\Gamma = \prod^6_{i=1}{\g_i(A_i)}$ where the $A_i$'s partition $E(G)$, and the $\g_i$'s are the six elements of $\fG$.  Then ${W_m}^{\Gamma}$ (or $(\bal, \bbe,\bga)^{\Gamma}$), the \emph{weight system permuted by $\Gamma$}, has the ordered triple of the weight system at a vertex $v_e$ given by $\eta ( \g_i) (\alpha_{v_e},\beta_{v_e},\gamma_{v_e})$ when $e \in A_i$.

\end{definition}

The classical Tutte polynomial has the duality property that the Tutte polynomial of a plane graph $G$ is the same as that of its geometric dual $G^*$ with the roles of the variables $x$ and $y$ permuted: $T(G;x,y)=T(G^*;y,x)$.  In~\cite{ES} it was shown that the topological transition polynomial has the duality property that 
\begin{equation}\label{transdual}
q(G_m; W_m, t) = q(G_m^*; W_m^*,t),
\end{equation}
\noindent
or equivalently, 
\begin{equation}
Q(G; (\boldsymbol\alpha, \boldsymbol\beta, \boldsymbol\gamma), t)=Q(G^*; (\boldsymbol\beta, \boldsymbol\alpha, \boldsymbol\gamma), t),
\end{equation}
 where $G^*$ is the geometric dual of $G$ and $W_m^*$ is the weight system that derives from exchanging the order of $\alpha_v$ and $\beta_v$ in the weight system at each vertex (\emph{i.e.} take $\Gamma = \delta(E(G))$ in Definition~\ref{permute weights}).  This led to a new duality result for the topological Tutte polynomial  in~\cite{ES}.  
 
 We are now able to extend this geometric duality to a full twisted duality property for the topological transition polynomial. This twisted duality relation says that the topological transition polynomial of the medial graph of $G$ is the same as that of the medial graph of any of the twisted duals, provided the weight system is appropriately permuted.  We   apply this twisted duality relation to derive new duality properties for the Penrose and topological Tutte polynomials in \cite{E-MMc} and \cite{E-MMd}.

\begin{theorem}\label{t.qsd}
Let $G$ be an embedded graph with embedded medial graph $G_m$, and let $\Gamma = \prod^6_{i=1}{\g_i(A_i)}$ where the $A_i$'s partition $E(G)$, and the $\g_i$'s are the six elements of $\fG$. Then,

\[ 
q\left(G_m; W_m, t\right) = q\left(G_m^{\Gamma}; W_m^{\Gamma}, t\right),  
\] 
\noindent
or equivalently,
\[
Q(G; (\boldsymbol\alpha, \boldsymbol\beta, \boldsymbol\gamma), t)=Q(G^{\Gamma}, (\boldsymbol\alpha, \boldsymbol\beta, \boldsymbol\gamma)^{\Gamma}, t).
\]
\end{theorem}
\begin{proof}
It suffices to prove this for one edge at a time, \emph{i.e.} to prove the property for $G^{\g(e)}$ where $\g \in \fG$.  The effect of $\g(e)$ on the medial graph is shown in the table below.  The first row gives $\g$.  The second row shows the effect on the arrow presentation of the arrows labelled $e$ in the arrow  presentation of $G^{\g(e)}$.  The labels 
$*, a, b, c$ 
are the labels on the disc(s) on either side of $e$, and the configurations of the discs are otherwise identical (so that, for example, if
\raisebox{-6mm}{\includegraphics[height=12mm]{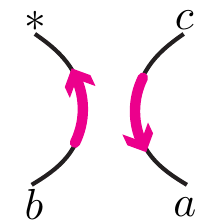}}
 represents two discs, then \raisebox{-6mm}{\includegraphics[height=12mm]{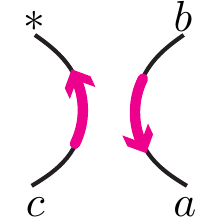}} represents one disc, and vice versa).  The third row shows the changes in the cyclic order of the vertices about $v_e$ in $(G^{\g(e)})_m$.
\[
\begin{tabular}{rcccccc}
$\g:$ & $1$ & $\tau$ & $\delta$ & $\tau \delta$ & $\delta \tau$ & $\tau\delta\tau$ 
\\  \\ \raisebox{7mm}{$G^{\g(e)}:$} & \includegraphics[height=16mm]{s1} & \includegraphics[height=16mm]{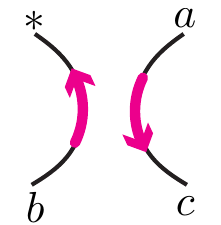}& \includegraphics[height=16mm]{s3}& \includegraphics[height=16mm]{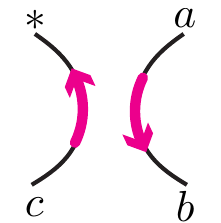}& \includegraphics[height=16mm]{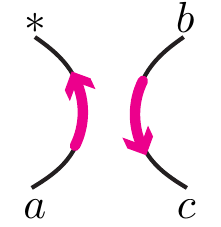}& \includegraphics[height=16mm]{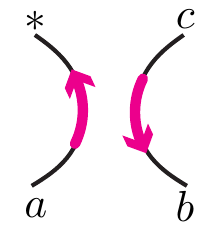}  \\
 \raisebox{7mm}{$(G^{\g(e)})_m:$} & \includegraphics[height=16mm]{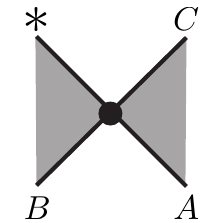} & \includegraphics[height=16mm]{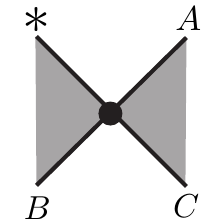}& \includegraphics[height=16mm]{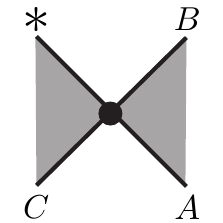}& \includegraphics[height=16mm]{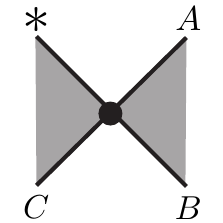}& \includegraphics[height=16mm]{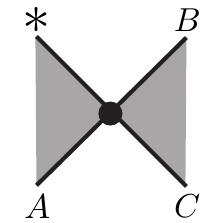}& \includegraphics[height=16mm]{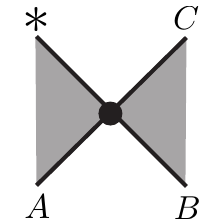}
\end{tabular}
\]
Note that if we give the labels $a, b, c$ the order $(a, b, c)$, then they 
are permuted in the second row of the table according to $\eta(\g)$, and hence $(\alpha_{v_e},\beta_{v_e},\gamma_{v_e})$ are permuted as claimed.  We illustrate this for $G^{\tau \delta (e)}$ below, leaving the other cases to the reader.
\[
\begin{tabular}{rccc}
 \raisebox{7mm}{$G_m \;\;=$} & \includegraphics[height=16mm]{s7} & \raisebox{7mm}{} & \includegraphics[width=50mm]{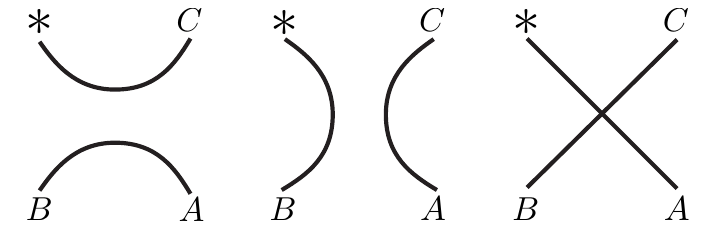}
 \\ & & &  \includegraphics[width=50mm]{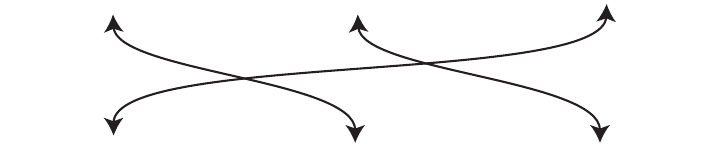}
 \\  \raisebox{7mm}{$(G^{\tau\delta(e)})_m \;\;=$} & \includegraphics[height=16mm]{s10} && \includegraphics[width=50mm]{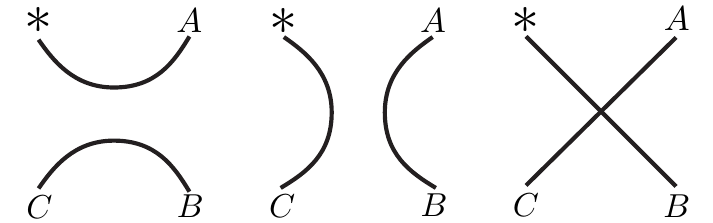}
\end{tabular}
\]

\end{proof}

Note that by taking take $\Gamma = \delta(E(G))$, the result of Equation~\ref{transdual} is now just an immediate corollary of Theorem~\ref{t.qsd}. 

\begin{corollary}

$Q(G; (\boldsymbol1, \boldsymbol1,\boldsymbol1),t)$ is constant on $Orb(G)$, where $(\boldsymbol1, \boldsymbol1,\boldsymbol1)$ is the weight system that assigns a 1 to every vertex state.
\end{corollary}
\begin{proof}
This follows from Theorem~\ref{t.qsd} since $  (\boldsymbol1, \boldsymbol1,\boldsymbol1)= (\boldsymbol1, \boldsymbol1,\boldsymbol1)^\Gamma$ for all $\Gamma$.
\end{proof}

Note that $Q(G; (\boldsymbol1, \boldsymbol1,\boldsymbol1),t)$ is just the generating function for the number of Eulerian $k$-partitions (see \cite{Las83}) of $G_m$.  It is an open question to characterize graphs such that $Q(G; (\boldsymbol1, \boldsymbol1,\boldsymbol1),t)=Q(H; (\boldsymbol1, \boldsymbol1,\boldsymbol1),t)$ when $G$ and $H$ are not twisted duals of one another.

\begin{corollary}\label{c.trans delete contract} Let $G$ be an embedded graph, $e \in E(G)$ and   $\g \in \fG$. Further, let $(\bal,\bbe,\bga)^{\g(e)}=(\bal',\bbe',\bga')$. Then
\[
Q(G; (\boldsymbol\alpha, \boldsymbol\beta, \boldsymbol\gamma), t)=\alpha_e' Q(G^{\g(e)}/e; (\boldsymbol\alpha', \boldsymbol\beta', \boldsymbol\gamma'), t)+\beta_e' Q(G^{\g(e)}-e; (\boldsymbol\alpha', \boldsymbol\beta', \boldsymbol\gamma'), t)+\gamma_e' Q(G^{\tau\g(e)}/e; (\boldsymbol\alpha', \boldsymbol\beta', \boldsymbol\gamma'), t).
\]
\end{corollary}
\begin{proof}
The identity follow easily from  Theorems~\ref{t.trans delete contract} and~\ref{t.qsd}.
\end{proof}

We have shown throughout this paper that there are deep and meaningful connections between the concepts of twisted duals and of medial graphs. In particular, in this section we have shown how these connections can be used to develop the theory of graph polynomials  by showing that the ribbon group action induces an action on the weights of the transition polynomial and obtaining a deletion-contraction-type relation for the transition polynomial. The applications of twisted duality to graph polynomials can in fact be pushed much further. Twisted duality can be used to obtain a host of new properties of other important graph polynomials. The topological transition polynomial assimilates a number of topological graph polynomials, consequently, the tools of twisted duality can be used to unravel the structure of these graph polynomials. We explore and develop applications of twisted duality to the Penrose polynomialand to the  topological Tutte polynomial   in the papers \cite{E-MMc} and \cite{E-MMd}.
 In \cite{E-MMc}, we introduce the topological Penrose polynomial 
 of an embedded graph $G$ and show that it can be recovered from the transition polynomial. In \cite{ES}, it was shown that the topological Tutte polynomial agrees with the topological transition polynomial on a restricted set of variables.
In \cite{E-MMc} and \cite{E-MMd} we exploit the connection between these graph polynomials and the transition polynomial, as well as the behaviour of the transition polynomial under the ribbon group action to find a number of properties of, and relationships among, various graph polynomials. These applications provide further evidence that twisted duality has the scope to develop the understanding of a wide variety of graph theoretical problems.

\section{Further directions}\label{conclusion}

Although the theory of twisted duality has answered many questions, many more have arisen in the course of this investigation.  We consolidate here some questions mentioned in previous sections, and add a few others.   

\begin{enumerate}

\item We have seen  that twisted duality is intimately related to medial graphs and checkerboard colourability, but further depths remain.  For example, if $G$ is any embedded $4$-regular graph, which of its twisted duals are also $4$-regular and checkerboard colourable?  These would be the twisted duals that are actually the embedded medial graphs of some embedded graph. Suppose $G$ and $H$ are non-isomorphic twisted duals that are both checkerboard colourable. How are the cycle family graphs of $G$ and $H$ related?  Also, is it possible to characterize those embedded graphs, without degree restrictions, that have a checkerboard colourable twisted dual?  Clearly every bipartite graph does (take its full dual), and this question is equivalent to characterizing those $G$ such that $Orb(G)$ contains a bipartite graph.

\item In this paper we investigated orbits of the ribbon group action, in particular determining the orbits of some special subgroups.  However, there are other subgroups of the ribbon group whose actions would be worth exploring.  Perhaps more importantly, it remains to understand various stabilizer subgroups of the ribbon group action. Is it possible to characterize embedded graphs that are ``self twisted dual'', that is, to determine  some $\Gamma=\prod_{i=1}^6\g_i (A_i)$, so that
$G^{\Gamma}$ is equivalent to $G$ as embedded graphs, locally embedded maps or abstract graphs?

\item A fundamental question that we have not addressed here is how topological invariants of an embedded graph $G$, such as  orientability, genus, degree sequence, or number of vertices, vary over the elements in an orbit of the ribbon group action. Is it possible to determine ranges for any of these invariants?  More ambitiously, is it possible to determine any of these invariants for $G^{\vg}$ just from the value of the invariant on $G$ and the knowledge of $\vg$?

\item Although we focused mainly on transition polynomial here, questions also abound for classical graph polynomials such as the Tutte or chromatic polynomials. In particular, how are various graph polynomials constrained on $Orb(G)$ for a fixed $G$?

\item 
Part of our interest in these problems derives from developing design strategies for self-assembling DNA nanostructures. Because of the great promise of nanotechnology, especially for biomolecular computing, but also, for example, drug delivery and biosensors (see \cite{Y+03, LL07}), recent research has focused on DNA self-assembly of nanoscale geometric constructs, notably graphs.  An essential step in building a self-assembling DNA nano-construct is designing the component molecules.  Ribbon graphs provide a particularly apt model for this application, with the edge boundaries representing single DNA strands and the vertices representing branched junction molecules.  Thus, the theory developed here for classifying graphs by their common medial graph suggests a possible efficient construction technique.  If a $4$-regular graph $F$ can be assembled out of DNA strands, with relatively small faces, then it might be used as a template to construct all of the graphs in $\mathcal{C}(F)$. The branched junction molecules forming its vertices would be decomposed into various vertex states, giving a state $\vec{s}$ of $F$. (This kind of split at the vertices was used to find Hamilton circuits by Adelman~\cite{Adl94}.) Then long (relative to the faces of $F$) double strands of DNA with cohesive ends (\emph{i.e.}  extended single strands of unsatisfied bases) might be introduced to form the edges of $F_{\vec{s}}$, with the relatively small cycles of $\vec{s}$ as the vertices.   This allows one $4$-regular graph to act as a single molecular ``template'' for constructing a whole class of nanostructures.  A number of graph polynomials, which are generally intractable to compute with current technology, might then, in principle, be found by a biomolecular computing process of chemical manipulations (splittings, cutting, and merging) of the molecules.

\end{enumerate}

\section*{Acknowlegements}

We thank Sergei Chmutov for pointing out to us the connection between  graphs of the form $G^{\tau(E(G))}$ and Petrie duals.

\bibliographystyle{amsplain} 

\end{document}